\RequirePackage{fix-cm}
\documentclass[smallextended]{svjour3}       
\smartqed  

\usepackage{appendix, amsmath,latexsym,mathrsfs,amsfonts,epsfig,color,graphicx,}

\renewcommand{\a}{\alpha}

\renewcommand{\l}{\left}
\renewcommand{\r}{\right}

\newcommand{\e}{\epsilon}
\newcommand{\eps}{\varepsilon}

\begin{document}

\title{Effective reduction for a nonlocal Zakai stochastic partial differential equation in data assimilation\thanks{This work was partly supported by the National Natural Science Foundation of China (NSFC)$11771161$ and $11771449.$}
}

\titlerunning{Effective reduction for a nonlocal Zakai equation}        

\author{Li Lin \textsuperscript{1} \and
        Jinqiao Duan \textsuperscript{2} \and Meihua Yang\textsuperscript{3}}
\authorrunning{Li Lin et al.} 

\institute{$^{1}$ School of Mathematics and Statistics and Center for Mathematical Sciences, Huazhong University of Science and Technology,Wuhan, 430074, China. \\
              \email{linli@hust.edu.cn}           
           \and\\
           $^{2}$ Department of Applied Mathematics, Illinois Institute of Technology, Chicago, IL 60616, USA.
              \email{duan@iit.edu.}
              \and\\
              $^{3}$ Corresponding author. School of Mathematics and Statistics and Center for Mathematical Sciences, Huazhong University of Science and Technology,Wuhan, 430074, China.
              \email{yangmeih@hust.edu.cn}
}

\date{Received: date / Accepted: date}

\maketitle

\begin{abstract}
We study the effective reduction for a nonlocal stochastic partial differential equation with oscillating coefficients.  The nonlocal operator in this stochastic partial differential equation is  the generator of  non-Gaussian L\'{e}vy processes, with either \textbf{integrable} or \textbf{non-integrable} jump kernels. We examine  the limiting behavior of this equation as  a scaling parameter tends to zero,  and derive a reduced (local or nonlocal)  effective  equation. In particular, this work leads to an effective  reduction for a data assimilation system with  L\'{e}vy noise, by examining the corresponding nonlocal Zakai stochastic partial differential equation.  We show that the probability density for the reduced data assimilation system approximates that for the original system.
\end{abstract}

\section{Introduction}\label{s1}
Data assimilation is a work to extract system state information with the help of observations. The state evolution and the observations are usually under random fluctuations. The basic thought is to obtain the best estimate for the true system state, in terms of the probability distribution for the system state, given only some noisy observations of the system. It has applications in various fields \cite{Bain,GE}. Some authors considered the effects of the multiscale signal and observation processes by studying the Zakai equation. Park \emph{et al}. \cite{park} show that the probability density of the original system converges to that of the effective system, by Fourier analysis method. Imkeller \emph{et al.} \cite{pim} show a convergence for the optimal filter, by using backward stochastic differential equations and asymptotic techniques.
However, random fluctuations are often non-Gaussian in nonlinear systems, for example, in geosciences \cite{dit}, and biosciences \cite{lin}. L\'{e}vy motions are appropriate models for a class of important non-Gaussian processes with jumps or bursts. It is desirable to consider data assimilation when the system evolution is under L\'{e}vy motions in the multiscale context. In this situation, the Zakai equation of the nonlinear system is a nonlocal stochastic differential equation.

Recently, there are also some papers about effective reduction of nonlocal problems. Mengesha \emph{et al.} \cite{meng} and Piatnitski \emph{et al.} \cite{31}
showed that in the topology of
resolvent convergence the family of rescaled operator converges to a second order elliptic operator with constant coefficients.
Schwab \cite{SRW} studied a family of nonlinear nonlocal equations  that include but are not limited to the control of pure jump processes. There are also some works about the effective reduction of nonlocal parabolic equations.  Bardi and Cesaroni \cite{BCA}
studied a nonlocal parabolic Hamilton-Jacobi equation with superlinear growth in the gradient terms. 
Acevessanchez and Cesbron \cite{AAAA3} provided the
rigorous derivation of the macroscopic limit of a Vlasov-Fokker-Planck equation in which the Laplacian is replaced by a fractional Laplacian.

However, there are few works   dealing with    effective reduction for nonlocal stochastic partial differential equations.
In this present paper, we prove two main theorems  about such nonlocal effective reduction, in Section \ref{s4} and  Section \ref{s5} respectively. In Section \ref{s4} , the nonlocal operator in stochastic partial differential equation is the generator of a L\'evy process with \textbf{integrable  jump kernel} near $0$ and the kernel also has  finite variance. In Section \ref{s5}, the nonlocal operator is the generator of an $\alpha$-stable L\'evy process, with  \textbf{non-integrable  jump kernel}. In Section \ref{s4}, we will show that   a \emph{nonlocal} stochastic partial differential equation has an effective system that is a    \emph{local}  stochastic partial differential equation. This further leads to the effective reduction for a nonlocal Zakai equation, for a   data assimilation system under non-Gaussian L\'evy fluctuations. This implies that the nonlocal Zakai equation can be effectively approximated by a local Zakai equation, with benefits for   simulation and analysis of a class of non-Gaussian data assimilation systems. In Section \ref{s5}, we show that a  stochastic partial differential equation with fractional Laplacian operator has an effective system that is also a fractional stochastic partial differential equation.
This paper is organized as follows. In Section \ref{s2}, we recall  some function spaces, and present assumptions  and two theorems. In Section \ref{s3}, we  discuss main ideas  for the  proofs  of the theorems, but  in a simplified setting.  Then, in Section \ref{s4}, we  prove Theorem \ref{1.1} and in Section \ref{s5}, we   prove Theorem \ref{27}.  Finally, in Section \ref{s6} we apply these two theorems   to the Zakai equations for data assimilation with two types of non-Gaussian  L\'evy fluctuations.
\section{Statements of Main Results}\label{s2}
Through this paper, we always identify functions on $\mathbb{T}$ with their periodic extension to $\mathbb{R}$. In this section, we first provide the simple statement of the problem. Then, we recall some function spaces and impose some assumptions on the coefficients. At last, we state two main theorems  about effective reduction. Theorem \ref{1.1} is for the integrable jump kernel  and Theorem \ref{27} is for the non-integrable jump kernel.

\subsection{Problem Statement}
In Section \ref{s4}, we consider the effective reduction  for the following  \emph{nonlocal} stochastic partial differential equation (\textbf{heterogeneous system}) with a small  positive scale parameter $\epsilon$:
\begin{equation}\label{1}
\begin{cases}
du^\epsilon(t,x)=A^\epsilon u^\epsilon (t,x) dt+B^\epsilon u^\epsilon(t,x)dt+\sigma\l(\frac{x}{\epsilon}\r)u^\epsilon(t,x)dW_t,\\
u^\epsilon(0,x)=u_0(x),
\end{cases}
\end{equation}
where the initial datum  $u_0$ is in $L^2(\mathbb{R})$,  $W=\l(W(t)\r)_{t\in[0,T]}$ is a one dimensional Brownian motion in a probability space $(\Omega,\mathcal{F},\mathbb{P})$, $A^\epsilon$ and $B^\epsilon$ are linear operators of the forms:
\begin{equation}\label{ee1}
\begin{split}
(A^\epsilon u)(x)=a\l(\frac{x}{\epsilon}\r)u''(x)+\frac{1}{\epsilon}b\l(\frac{x}{\epsilon}\r)u'(x),\\
(B^\epsilon u)(x)=\frac{1}{\epsilon^3}\lambda\l(\frac{x}{\epsilon}\r)\int_{\mathbb{R}}
c\l(\frac{x-y}{\epsilon}\r)\Big(u(y)-u(x)\Big)dy.
\end{split}
\end{equation}
Here $a(\cdot), b(\cdot), \sigma(\cdot)$,  $\lambda(\cdot)$ are known functions of period $1$, $\lambda(\cdot)$ is also  bounded and positive, and $c(z)$ is the jump kernel
being a positive \textbf{integrable function} with symmetry property $c(z)=c(-z)$.
That is, $\int_{\mathbb{R}} c(\xi) d\xi <\infty$.
For convenience, we also  define the linear operator $T^\epsilon u=A^\epsilon u+B^\epsilon u$.

Our purpose is to examine   the convergence of  the solution  $u^\epsilon$  of (\ref{1}) in some probabilistic sense,  as $\epsilon\rightarrow0$,
and to specify the limit $u$. We will see that the limit process $u$ satisfies the following \textbf{local} stochastic partial differential equation (\textbf{effective system}):
\begin{equation}\label{2}
  \begin{cases}
    du(t,x)=T^0u(t,x)dt+M^0u(t,x)dW_t,\\
    u(0,x)=u_0(x),
  \end{cases}
\end{equation}
where
\begin{equation}\label{17}
(T^0u)(x)=Q u''(x),\;\;\;\;\;(M^0u)(x)=u(x)\int_{\mathbb{T}}\sigma(\eta)m(\eta)d\eta.
\end{equation}
Here the coefficient $Q$ is determined by
\begin{equation}\label{14}
\begin{split}
  Q&=\int_{\mathbb{T}}a(\eta)m(\eta)\Big(\chi'(\eta)+1\Big)^2d\eta
  +\frac{1}{2}\int_{\mathbb{T}}\int_{\mathbb{R}}c(\eta-q)\lambda(q)m(q)\Big[(\eta-q)\\
  &\quad+\Big(\chi (\eta)-\chi (q)\Big)\Big]^2d\eta dq,
\end{split}
\end{equation}
and the functions $\chi$ and $m$ are the unique solutions of the following \emph{deterministic} partial differential equations, respectively:
\begin{equation}\label{9465}
  \begin{cases}
    \tilde{T}\chi(\eta)+b(\eta)=0,\;\eta\in\mathbb{T},\\
    \int_0^1\chi(\eta)m(\eta)d\eta=0,
  \end{cases}
\end{equation}
and
\begin{equation}\label{24}
  \begin{cases}
   \tilde{T}^* m(\eta)=0,\;\eta\in\mathbb{T},\\
   \int_0^1m(\eta)d\eta=1,
 \end{cases}
\end{equation}
with the linear operator $\tilde{T}$ being defined by:
\[
\begin{split}
(\tilde{T}\upsilon)(\eta):&=a(\eta)\upsilon''(\eta)+b(\eta)\upsilon'(\eta)+\lambda(\eta)\int_{\mathbb{R}}c(x-\eta)\Big(\upsilon(x)-\upsilon(\eta)\Big)dx,\eta\in\mathbb{T}.
\end{split}
\]
Note that $\tilde{T}^*$ is the adjoint operator of $\tilde{T}$ in $L^2(\mathbb{R})$.
For convenience, we denote
$$(\tilde{A}\upsilon)(\eta)=a(\eta)\upsilon''(\eta)+b(\eta)\upsilon'(\eta),$$
$$(\tilde{B}\upsilon)(\eta)=\lambda(\eta)\int_{\mathbb{R}}c(x-\eta)\Big(\upsilon(x)-\upsilon(\eta)\Big)dx.$$

We will apply the proceeding  result to obtain the effective equation for the following \textbf{nonlocal Zakai equation}:
\begin{equation}\label{3}
\begin{cases}
$$du^\epsilon(t,x)=(T^\epsilon)^* u^\epsilon(t,x)dt+u^\epsilon(t,x)\sigma\l(\frac{x}{\epsilon}\r)^2dt+u^\epsilon(t,x)\sigma\l(\frac{x}{\epsilon}\r)dW_t,\\
u^\epsilon(0,x)=u_0(x).$$
\end{cases}
\end{equation}
This is the Zakai equation for the conditional probability density of the following \textbf{nonlinear data assimilation system} (Qiao and Duan  \cite{32})
\begin{equation}
\begin{cases}
  dx_t^\epsilon=\frac{1}{\epsilon}b\l(\frac{x_t}{\epsilon}\r)dt+\sigma_1\l(\frac{x_t}{\epsilon}\r)dw_t+dL^{\epsilon}_t,\\
  dy_t=\sigma\l(\frac{x_t}{\epsilon}\r)dt+dW_t,
\end{cases}
\end{equation}
where $x_t$ is system (or signal) state, $y_t$ is the observation,  and $w^{\epsilon}_t, W_t$ are  mutually independent Brownian motions.  Moreover,
$L^{\epsilon}_t$ is a L\'{e}vy process with the generator $B^\epsilon$ given in (\ref{ee1}).

\medskip

In Section \ref{s5}, we consider the effective reduction  for the following  \emph{nonlocal} stochastic partial differential equation (\textbf{heterogeneous system}) with a small  positive scale parameter $\epsilon$:
\begin{equation}\label{21}
\begin{cases}
$$d v^\epsilon(t,x)=F^\epsilon v^\epsilon (t,x) dt+L^\epsilon v^\epsilon(t,x)dt+\sigma\l(\frac{x}{\epsilon}\r)v^\epsilon(t,x)dW_t,\\
v^\epsilon(0,x)=v_0(x),$$
\end{cases}
\end{equation}
where the initial datum  $v_0$ is in $L^2(\mathbb{R})$,  $W=(W(t))_{t\in[0,T]}$ is a one dimensional Brownian motion in a probability space $(\Omega,\mathcal{F},\mathbb{P})$, and for $\a\in (1,2),$  we define
the linear operators:

$$(F^\epsilon u)(x)=g\l(\frac{x}{\epsilon}\r)u'(x)+f\l(\frac{x}
{\epsilon}\r)u(x),$$
$$(L^\epsilon u)(x)=\int_{\mathbb{R}\backslash \{0\}}\Big(u(x+\delta\l(\frac{x}{\epsilon}\r)y)-u(x)\Big)\nu^{\alpha}(dy)+\frac{1}{\epsilon^{\alpha-1}}p\l(\frac{x}{\epsilon}\r)u'(x),$$
where the integral is in the sense of Cauchy principle value.
Here $p(\cdot), g(\cdot),$ $f(\cdot), \delta(\cdot)$ are known functions of period $1$, $\delta(\cdot)$ is positive  and the jump measure $\nu^{\alpha}(dy)=|y|^{-(1+\a)}dy.$
Note that this \textbf{jump kernel is non-integrable}: $\int_{\mathbb{R}} |y|^{-(1+\a)}dy  = \infty$.
We denote $V^\epsilon u=F^\epsilon u+L^\epsilon u$.

Recall that the nonlocal or fractional Laplacian is defined as \cite{11,HuangQiao}
$$(-\Delta)^{\a/2}u(x)=\int_{\mathbb{R}\backslash \{0\}}\frac{u(x)-u(y)}{|y-x|^{1+\a}}dy,$$
where the integral is in the sense of Cauchy principal value.
Note that $-(-\Delta)^{\a/2}$ is the generator for a symmetric $\alpha$-stable L\'evy motion $L_t^\alpha$. Thus  $L^\epsilon u(x)=-\delta^{\a}\l(\frac{x}{\epsilon}\r)(-\Delta)^{\a/2}u(x)+\frac{1}{\epsilon^{\alpha-1}}p\l(\frac{x}{\epsilon}\r)u'(x).$
Denoting $\delta_1(x)=\delta^\alpha(x)$.\\

\begin{remark}
  We introduce the nonlocal gradient operator and nonlocal divergence operator. Given the mapping $\beta(x,y),\gamma(x,y): \mathbb{R}\times\mathbb{R}\rightarrow\mathbb{R}$ with $\gamma$ antisymmetric, the action of the nonlocal divergence operator $\mathcal{D}$ on $\beta$ is defined as
$$\mathcal{D}(\beta)(x):=\int_{\mathbb{R}}(\beta(x,y)+\beta(y,x))\cdot\gamma(x,y)dy  \qquad \text{for}\; x\in \mathbb{R}.$$
Given the mapping $u(x):\mathbb{R}\rightarrow\mathbb{R},$ the adjoint operator $\mathcal{D^{*}}$ corresponding to $\mathcal{D}$ is the nonlocal gradient operator whose action on $u$ is given by
$$\mathcal{D^{*}}(u)(x,y)=-(u(y)-u(x))\gamma(x,y) \qquad \text{for}\;  x,y\in \mathbb{R}.$$
Here we take $\gamma(x,y)=(y-x)\frac{1}{|y-x|^{\frac{3+\a}{2}}}.$ Then
$$\mathcal{D}\mathcal{D^{*}}=2(-\Delta)^{\a/2}.$$
\end{remark}

We will examine   the convergence of  the solution  $v^\epsilon$  of (\ref{21}) in some probabilistic sense,  as $\epsilon\rightarrow0$,
and to specify the limit $v$.
We will see that the limit process $v$ satisfies the following \textbf{nonlocal} stochastic partial differential equation (\textbf{effective system}):
\begin{equation}\label{22}
  \begin{cases}
    dv(t,x)=V^0v(t,x)dt+M^0v(t,x)dW_t,\\
    v(0,x)=v_0(x),
  \end{cases}
\end{equation}
where
$$(M^0u)(x)=u(x)\int_{\mathbb{T}}\sigma(\eta)m_1(\eta)d\eta,$$
\[
\begin{split}
(V^0u)(x)&=\int_\mathbb{T}\delta^\alpha(\eta)m_1(\eta)d\eta\cdot(-(-\Delta)^{\alpha/2})u(x)+u'(x)\int_{\mathbb{T}}g(\eta)m_1(\eta)d\eta\\
&\quad+u(x)\int_{\mathbb{T}}f(\eta)m_1(\eta)d\eta.
\end{split}
\]

Here the function $m_1$ is the unique solution of the following \emph{deterministic} partial differential equation, respectively:
\begin{equation}\label{25}
  \begin{cases}
   \tilde{L}^* m_1(\eta)=0,\;\eta\in\mathbb{T},\\
   \int_0^1m_1(\eta)d\eta=1,
 \end{cases}
\end{equation}
with the linear operator $\tilde{L}$ being defined by:
\[
\begin{split}
(\tilde{L}u)(\eta):&=-\delta^{\a}(\eta)(-\Delta)^{\a/2}u(\eta)+p(\eta)u'(\eta),\eta\in\mathbb{T}.
\end{split}
\]
Note that $\tilde{L}^*$ is the adjoint operator of $\tilde{L}$ in $L^2(\mathbb{R})$.

We will also apply the result to obtain the effective equation for the following \textbf{nonlocal Zakai equation}:
\begin{small}
\begin{equation}
\begin{cases}
$$dv^\epsilon(t,x)=(L^\epsilon)^* v^\epsilon(t,x)dt+v^\epsilon(t,x)\sigma\l(\frac{x}{\epsilon}\r)^2dt+v^\epsilon(t,x)\sigma\l(\frac{x}{\epsilon}\r)dW_t,\\
v^\epsilon(0,x)=v_0(x).$$
\end{cases}
\end{equation}
\end{small}
This is the Zakai equation for the conditional probability density of the following \textbf{nonlinear data assimilation system}
\begin{equation}\label{nn}
\begin{cases}
  dx_t^\epsilon=\frac{1}{\epsilon}p\l(\frac{x_t}{\epsilon}\r)dt+\delta{\l(\frac{x_t}{\epsilon}\r)}dL^\alpha_t,\\
  dy_t=\sigma(\frac{x_t}{\epsilon})dt+dW_t,
\end{cases}
\end{equation}
where $x_t$ is system (or signal) state, $y_t$ is the observation,  and $W_t$ is a one dimensional Brownian motion.  Moreover,
$L^\alpha_t$ is a $\alpha$-stable L\'{e}vy process.

\medskip


The major difference between (\ref{1}) and (\ref{21}) is the different scaling in
the generator of the L\'evy process.
 In Section \ref{s4}, the operator in stochastic partial differential equation (\ref{1}) is the generator of the L\'evy process with special jump kernel which excludes the big jumps.  The scaling in the generator is  $\epsilon^{-3}$ . In this case, the operator in effective system (\ref{2}) is just a generator of Brownian motion. In Section \ref{s5}, the operator in stochastic partial differential equation (\ref{21}) is the generator of a multiplicative $\alpha$-stable L\'evy process which includes big jumps.
 There is no explicit term about the parameter $\epsilon$ in the fractional term.
  Then we can find the operator in effective system (\ref{22}) is also a generator of a $\alpha$-stable L\'evy process.
\subsection{Function Spaces and Assumptions}\label{ss1}
\subsubsection{Function Spaces on $\mathbb{R}$}
Let us set $H=L^2(\mathbb{R})$, the totality of square integrable functions on $\mathbb{R}$ with canonical inner product and norm
$$(u,v)=\int_{\mathbb{R}}u(x)v(x)dx, \;\left\|u\right\|^2_0=(u,u),\; u,v\in L^2(\mathbb{R}).$$
We denote by  $H^1=H^1(\mathbb{R})$, the usual Sobolev space of order $1,$ that is, the completion of $C_c^\infty(\mathbb{R}),$ the set of smooth functions with compact support, with respect to the norm $\left\|\cdot\right\|_1$
induced by the inner product
$$(u,v)_1=(u,v)+(u',v'),$$
where $f'$ stands for  $\frac{df}{dx}.$ We denote by $(H^1)'=H^{-1}$ the dual space of $H^1.$


We fix a smooth and positive function $\theta$ on $\mathbb{R}$ such that $\theta(x)=|x|$ for all $|x|\geq1$. Then for $n=0$ or $1,$ $\kappa\in[0,\infty)$, we define the weighted Sobolev space with norm $\left\|\cdot\right\|_{H^n_\kappa}$ :
$$H^n_\kappa=\{\upsilon|\upsilon e^{\kappa\theta}\in H^n_\kappa\},\qquad \left\|\upsilon\right\|_{H^n_\kappa}=\left\|\upsilon e^{\kappa\theta}\right\|_n.$$
Let $(H^n_\kappa)'$ be the dual space of $H^n_\kappa$. Then
$$(H^n_\kappa)'=H^{-n}_{-\kappa}=\{\upsilon|\upsilon e^{-\kappa\theta}\in H^{-n}\}.$$
We shall denote by $\left \langle , \right \rangle$ the duality product between $H^n_\kappa$ and $H^{-n}_{-\kappa}.$

We also denote $H_{-\kappa} = H_{-\kappa}^0$.
Introduce a function space
$$ K=C(0,T;H_{-\kappa}^{-1})\cap L^2(0,T;H_{-\kappa}).$$
In $C(0,T;H_{-\kappa}^{-1})$ we take the uniform convergence topology  $\mathcal{T}_1$,  and in $L^2(0,T;\\
H_{-\kappa})$ we choose the topology induced by $L^2$-norm and denote it by  $\mathcal{T}_2$.
\begin{remark}
  Two topologies $\mathcal{T}_1$ and $\mathcal{T}_2$ generate the same Borel $\sigma-$algebra on $K$(p.$74$ of \cite{MMM}).
\end{remark}
We recall some notations used in Section \ref{s5}. For any $\a\in (1,2)$, we define the usual fractional space $H^{\a/2}(\mathbb{R})$, that is, the completion of $C_c^\infty(\mathbb{R}),$  with respect to the following norm:
\begin{equation*}
\begin{split}
  \left\|u\right\|_{H^{\a/2}}^2&=\left\|u\right\|_0^2+\int_{\mathbb{R}}\int_{\mathbb{R}}\frac{|u(x)-u(y)|^2}{|x-y|^{1+\a}}dxdy\\
  &:=\left\|u\right\|_0^2+[u]^2_{H^{\alpha/2}}.
\end{split}
\end{equation*}
Then, we can do as before, for $\a\in(0,2)$, $\kappa\in[0,\infty)$, we define the weighted Sobolev spaces $H^{\a/2}_\kappa$, $H^{-\a/2}_{-\kappa}$ and $ K_1=C(0,T;H_{-\kappa}^{-\a/2})\cap L^2(0,T;H_{-\kappa}).$
\subsubsection{Assumptions}
Now, what we impose on the coefficients are the following conditions.\\
($\romannumeral 1$) The coefficients $a(\cdot), b(\cdot)\in C^3 (\mathbb{T}), \sigma(\cdot)\in C_b^3(\mathbb{T}), $ where $C_b^3$ stands for  the set of function of class\;$C^3$\;whose partial derivatives of order less than or equal to\;$3$\; are bounded.\\
($\romannumeral 2$) For every $\eta\in \mathbb{T}$, there exists $\kappa_1>0$\;such that\\
$$\kappa_1  \leq a(\eta)\leq\kappa_1^{-1}.$$
($\romannumeral 3$)  The function $\lambda(x)$ is periodic  and bounded (with bounds $C_1, C_2$ ): $0<C_1\leq\lambda(x)\leq C_2<\infty.$\\
($\romannumeral 4$)  The  kernel function $ c(z)\geq0; c(-z)=c(z)$, and $$\left\|c\right\|_{L^1({\mathbb{R})}}=\int_{\mathbb{R}}c(z)dz=a_1>0,
\int_{\mathbb{R}}\vert z\vert ^2 c(z)dz<\infty.$$
($\romannumeral 5$)  The function $b(\cdot)$ satisfies
$$\int_{\mathbb{T}}b(\eta)m(\eta)d\eta=0,$$
where $m$ is the solution of $(\ref{24})$.\\

In Section \ref{s5}, we  make the following  \textbf{assumptions}. \\
(a) The coefficients $p(\cdot), g(\cdot), f(\cdot), \delta(\cdot),\sigma(\cdot)\in C^2 (\mathbb{T}).$\\
(b) Function $p(\cdot)$ satisfies:
$$\int_{\mathbb{T}}p(\eta)m_1(\eta)d\eta=0,$$
where $m_1$ is the solution of (\ref{25}).

\subsection{Main Result}
\begin{theorem}\label{1.1}(Effective reduction for integrable jump kernel)
  Assume that assumptions $(\romannumeral 1)-(\romannumeral 5)$ hold. Let $u^\epsilon$ be the solution of  the heterogeneous equation (\ref{1}). We denote $\pi^{m,\epsilon},$ $\pi^\epsilon$ the laws of $m^\epsilon u^\epsilon$and $u^\epsilon$ respectively, where we have set $m^\epsilon(x)=m(\frac{x}{\epsilon})$, $(m^\epsilon u^\epsilon)(x)=m(\frac{x}{\epsilon})u^\epsilon(x),$ and $m$ is the solution of the equation (\ref{24}). Then, we have
  $$\pi^{m,\epsilon}\Rightarrow\pi\quad in\quad(K,\mathcal{T}_1),\quad\quad\pi^\epsilon\Rightarrow\pi\quad in\quad(K,\mathcal{T}_2)$$
  as $\epsilon\rightarrow 0,$ where $\pi$ is the probability law on $K$ induced by the
  solution of equation (\ref{2}) and $K$ is defined in Subsection $\ref{ss1}.$
\end{theorem}

\begin{theorem}\label{27}(Effective reduction for non-integrable jump kernel)
  Assume that assumptions $(a)$ and $(b)$  hold. Let $v^\epsilon$ be the solution of  the heterogeneous equation (\ref{21}). We denote $\pi_1^{m_1,\epsilon},$ $\pi_1^\epsilon$ the laws of $m_1^\epsilon v^\epsilon$and $v^\epsilon$ respectively, where we have set $m_1^\epsilon(x)=m_1(\frac{x}{\epsilon})$, $(m_1^\epsilon v^\epsilon)(x)=m_1(\frac{x}{\epsilon})v^\epsilon(x),$ and $m_1$ is the solution of the equation (\ref{25}). Then, we have
  $$\pi_1^{m_1,\epsilon}\Rightarrow\pi_1\quad in\quad(K_1,\mathcal{T}_1),\quad\quad\pi_1^\epsilon\Rightarrow\pi_1\quad in\quad(K_1,\mathcal{T}_2)$$
  as $\epsilon\rightarrow 0,$ where $\pi_1$ is the probability law on $K_1$ induced by the
  solution of equation (\ref{22}) and $K_1$ is defined in Subsection $\ref{ss1}.$
\end{theorem}

\begin{remark}
 In Theorem \ref{1.1} and Theorem \ref{27}, we take the convergences $\pi^{m,\epsilon}\Rightarrow\pi\;\text{in}\;(K,\mathcal{T}_1)$ and
 $\pi_1^{m_1,\epsilon}\Rightarrow\pi_1\; \text{in}\;(K_1,\mathcal{T}_1)$ as conclusions. We can apply them to the proof of the convergence of  the Zakai equations in  Section \ref{s6}.

\end{remark}

We now give two lemmas  about the well-posedness for  the heterogeneous equation (\ref{1}) and and the effective equation  (\ref{2}). Lemma \ref{2.1} will be proved in Appendix A.

\begin{lemma}\label{2.1}(Well-posedness for heterogeneous equation)\\
 There exists a unique mild solution $u^\epsilon_t\in C(0,T;H)\cap L^2(0,T;H^1)$ of the heterogeneous equation (\ref{1}).
\end{lemma}
   The existence and uniqueness of  the solution for  the effective equation (\ref{2}) is well known (Pardoux \cite{27}).
\begin{lemma}\label{2.2}(Well-posedness for effective equation)\\
There exists a solution $u\in L^2([0,T];H^1)$ for the effective equation (\ref{2}).
It is unique in the sense:
  $\mathbb{P}\Big(u=v$ in $H^{-1}$,\; $\forall t\in[0,T]\Big)=1$, for every $u,v$ satisfying the equation.
\end{lemma}

Next, we provide a brief description of the existence and uniqueness of the weak solutions for equations (\ref{21}) and (\ref{22}).

We add  terms $\epsilon_1\Delta v^\e(t,x)$ and $\epsilon_1\Delta v(t,x)$ to the right hands of the equations (\ref{21}) and (\ref{22}) respectively. The well-posedness of the new equations follow by \cite{27}. Then by a vanishing viscosity method (\cite{HJC}), we obtain
the existence and uniqueness of the weak solutions for equations (\ref{21}) and (\ref{22}).

\section{Ideas of Proof in a Simplified Setting}\label{s3}

Here we outline  the ideas for the  proving  Theorem \ref{1.1}, but in a simplified setting.

Denote $\eta=\frac{x}{\epsilon}$ a variable on the period: $\eta\in \mathbb{T}$. Let  $a(x)=b(x)=0$ and $\lambda(x)=\sigma(x)=1.$ Then from equations (\ref{9465}) and (\ref{24}), we have $\chi(\eta)=0, m(\eta)=1$ respectively. We can provide an outline of our argument for the following equation:
\begin{equation}\label{864}
\begin{cases}
du^\epsilon(t,x)=B^\epsilon u^\epsilon(t,x)dt+u^\epsilon(t,x)dW_t,\\
u^\epsilon(0,x)=u_0(x),
\end{cases}
\end{equation}
and show the following is the effective equation:
\begin{equation}\label{865}
\begin{cases}
du(t,x)=B^0 \Delta u(t,x)dt+u(t,x)dW_t,\\
u(0,x)=u_0(x),
\end{cases}
\end{equation}
where $B^0=\frac{1}{2}\int_{\mathbb{R}}z^2c(z)dz.$

At first, from Lemma \ref{2.3} and Lemma \ref{2.4}, we conclude the uniform estimates and equicontinuity of the solution $u^\epsilon$ of equation (\ref{864}) respectively.  We can obtain the tightness
of the probability law $\pi^\epsilon$ induced by the solution $u^\epsilon$ in $(K,\mathcal{T}_1)$.
By the Prohorov theorem, there exists a subsequence $\epsilon_k$ and a probability measure $\widetilde {\pi}$ such that $\pi^{\epsilon_k}\rightarrow\widetilde {\pi}$ in $(K,\mathcal{T}_1),$
as $\epsilon$ goes to $0.$
Then we have $\pi^{\epsilon_k}\rightarrow\widetilde {\pi}$ in $(K,\mathcal{T}_2)$ by Lemma \ref{2.7}.

Next, we verify that $\widetilde {\pi}$ coincides with the law induced by the solution of the effective equation (\ref{865}).
By Lemma \ref{3542}, we just need to show the formula (\ref{8}) goes to $0$, as $\epsilon\rightarrow 0$ .
The most important and difficult part in this step is constructing  a family of test functions  $\xi_\epsilon$  such that $(B^\epsilon)^*\xi_\epsilon\rightarrow B^0\xi'',$ as $\epsilon$
goes to $0.$  Let us define $\xi^{\epsilon}\in C_c^\infty(\mathbb{R})$ by
 \begin{equation*}
 \xi^{\epsilon}(x)=\xi(x)+\epsilon h_{1}\l(\frac{x}{\epsilon}\r)\xi'(x)+\epsilon^2 h_2 \l(\frac{x}{\epsilon}\r)\xi''(x).
 \end{equation*}
 Through simple calculations, let $z=\frac{x-y}{\epsilon}$ we have
 \[
 \begin{split}
 (B^\epsilon)^*\xi(x)&=\frac{1}{\epsilon^2}\int_{\mathbb{R}}c(z)\l[\l(-\epsilon z\r)\xi'(x)+\frac{1}{2}(\epsilon z)^2\xi''(x)\r]dz+o(\epsilon)\\
 &=-\frac{1}{\epsilon}\l(\int_{\mathbb{R}}zc(z)dz\r)\xi'(x)+\frac{1}{2}\l(\int_{\mathbb{R}}z^2c(z)dz\r)\xi''(x)+o(\epsilon).
 \end{split}
 \]

 \[
 \begin{split}
 \epsilon(B^\epsilon)^*\l(\xi'(x)h_1(x)\r)&=\l(\frac{1}{\epsilon}\int_{\mathbb{R}}c(z)\l[h_1\l(\frac{x}{\epsilon}-z\r)-h_1\l(\frac{x}{\epsilon}\r)\r]dz\r)\xi'(x)\\
 &\quad-\l(\int_{\mathbb{R}}h_1\l(\frac{x}{\epsilon} -z\r)zc(z)dz\r)\xi''(x)+o(\epsilon).
 \end{split}
 \]
 \begin{small}
 $$\epsilon^2(B^\epsilon)^*\l(\xi''(x)h_2(x)\r)=\l(\int_{\mathbb{R}}\l[h_2\l(\frac{x}{\epsilon}-z\r)-h_2\l(\frac{x}{\epsilon}\r)\r]c(z)dz\r)\xi''(x)+o(\epsilon).$$
\end{small}
Then, we assume that the average of each component of functions $h_1(\eta), h_2(\eta)$ over the period is equal to $0,$ and we set
$$\int_{\mathbb{R}}c(x-\eta)\l(h_1(x)-h_1(\eta)\r)dx=\int_{\mathbb{R}}zc(z)dz=0,$$
$$\int_{\mathbb{R}}c(x-\eta)\l(h_2(x)-h_2(\eta)\r)dx=\int_{\mathbb{R}}h_1(\eta)zc(z)dz.$$
We can obtain that $(B^\epsilon)^*\xi_\epsilon\rightarrow \frac{1}{2}\l(\int_{\mathbb{R}}z^2c(z)dz\r)\xi'',$ as $\epsilon$ goes to $0.$
The result is consistent with the equation (\ref{14}).
Then $\widetilde {\pi}$ is the law induced by the solution of the effective equation (\ref{865}) by Lemma \ref{2.11}.



\section{Proof of Theorem 1}\label{s4}
 The proof of \textbf{Theorem \ref{1.1}} is  divided into two steps: the tightness and the limit law.

\subsection{Tightness}
We denote $\pi^{m,\epsilon}$   the probability measure induced by $m^\epsilon u^\epsilon$. we will show the tightness of $\{\pi^{m,\epsilon}:  \epsilon >0\}$ in $(K,\mathcal {T}_1)$.
    \begin{lemma}\label{2.5}
      If $u^\epsilon$ is a solution of heterogeneous equation (\ref{1}), and $\pi^{m,\epsilon}$ is the probability measure induced by $m^\epsilon u^\epsilon$. Then,
      $\{\pi^{m,\epsilon}:\epsilon>0\}$ is tight in $(K,\mathcal{T}_1)$.
\end{lemma}

To prove Lemma \ref{2.5}, we begin with the following two lemma.

\begin{lemma}\label{2.3}
    Let $u_t^\epsilon$\;be a solution of heterogeneous equation (\ref{1}) with initial value $u_0\in H$. Then there exists a positive constant C,  independent of $\epsilon$, such that
    \begin{equation}\label{88}
    \sup_\epsilon \mathbb{E}\l[\sup_{0\leq t\leq T}\left\|u^\epsilon_t\right\|_0^4\r]+\sup_\epsilon \mathbb{E}\l[\l(\int_0^T\left\|u^\epsilon_t\right\|^2_1 dt\r)^2\r]\leq C\l(1+\left\|u_0\right\|^4_0\r).
    \end{equation}
  \begin{proof}
    First we assume that $u_0$ is in $C_c^\infty(\mathbb{R})$, and choose a version of $u^\epsilon$ such that $u^\epsilon_t(x)\in C_0^2([0,T]\times\mathbb{R}),
    \mathbb{P}$-almost surely. This is possible under the assumptions $(\romannumeral 1)-(\romannumeral 4)$. By It\^{o}'s formula, we have
    \begin{equation*}
    \begin{split}
    u^\epsilon_t(x)^2&=u_0(x)^2+2\int_0^t(T^\epsilon u^\epsilon_s)(x)u^\epsilon_s(x)ds\\&\quad+\int_0^t\Big(u^\epsilon_s(x)\sigma\l(\frac{x}{\epsilon}\r)\Big)^2ds
    +2\int_0^t u^\epsilon_s(x)^2\sigma\l(\frac{x}{\epsilon}\r)dW_s.
    \end{split}
    \end{equation*}
    This  is equivalent to
    \begin{equation}\label{5}
    \begin{split}
    m^\epsilon(x)u^\epsilon_t(x)^2&=m^\epsilon(x)u_0(x)^2+2\int_0^t(T_m^\epsilon u^\epsilon_s)(x)u^\epsilon_s(x)ds\\
    &\quad+\int_0^tm^\epsilon(x)\Big(u^\epsilon_s(x)\sigma\l(\frac{x}{\epsilon}\r)\Big)^2ds\\
    &\quad+2\int_0^tm^\epsilon(x) u^\epsilon_s(x)^2\sigma\l(\frac{x}{\epsilon}\r)dW_s,
    \end{split}
    \end{equation}
    where $m^\epsilon(x)=m(\frac{x}{\epsilon})$,  and $T^\epsilon_m=m^\epsilon T^\epsilon$ is defined by:
    \begin{equation}
      \begin{split}
    (T^\epsilon_mu)(x)&=a^m\l(\frac{x}{\epsilon}\r)u''(x)+\frac{1}{\epsilon}b^m\l(\frac{x}{\epsilon}\r)u'(x)\\
    &\quad+\frac{1}{\epsilon^3}\lambda^m\l(\frac{x}{\epsilon}\r)\int_{\mathbb{R}}c\l(\frac{x-y}{\epsilon}\r)\Big(u(y)-u(x)\Big)dy\\&=
    \Big(a^m\l(\frac{x}{\epsilon}\r)u'(x)\Big)'+\frac{1}{\epsilon}\beta^m\l(\frac{x}{\epsilon}\r)u'(x)\\
    &\quad+
    \frac{1}{\epsilon^3}\lambda^m\l(\frac{x}{\epsilon}\r)\int_{\mathbb{R}}c\l(\frac{x-y}{\epsilon}\r)\Big(u(y)-u(x)\Big)dy.
      \end{split}
    \end{equation}
    We introduce notations:  $a^m(\eta)\!=\!a(\eta)m(\eta), b^m(\eta)\!=\!m(\eta)b(\eta), \lambda^m(\eta)\!=\!\lambda(\eta)\cdot\\
    m(\eta)$, and $$\beta^m(\eta)=b^m(\eta)-a^m(\eta)'.$$
    Then from the fact that  $(\tilde{T})^*m(\eta)=0$, we obtain
    \begin{equation}
      \begin{split}
    &\Big(B_m^\epsilon u+\frac{1}{\epsilon}\beta^m\l(\frac{x}{\epsilon}\r)u',u\Big)\\&=\Big(-\frac{1}{2\epsilon^2}(\beta^m)^{'}\l(\frac{x}{\epsilon}\r),u^2\Big)
    +\frac{1}{2\epsilon^3}\bigg(\lambda^m(\frac{x}{\epsilon}),\int_{\mathbb{R}}c(\frac{x-y}{\epsilon})\Big(u^2(y)-u^2(x)\Big)dy\bigg)
    \\&\quad+(B_m^\epsilon u,u)-\frac{1}{2\epsilon^3}\bigg(\lambda^m\l(\frac{x}{\epsilon}\r),\int_{\mathbb{R}}c\l(\frac{x-y}{\epsilon}\r)\Big(u^2(y)-u^2(x)\Big)dy\bigg)\\
    &=\frac{1}{2\epsilon^2}\Big((\tilde{T})^*m,u^2\Big)+\frac{1}{\epsilon^3}\int_{\mathbb{R}}\lambda^m\l(\frac{x}{\epsilon}\r)
    \int_{\mathbb{R}}c\l(\frac{x-y}{\epsilon}\r)\Big(u(y)-u(x)\Big)u(x)dydx\\
    &\quad-\frac{1}{\epsilon^3}\frac{1}{2}\int_{\mathbb{R}}\lambda^m\l(\frac{x}{\epsilon}\r)
    \int_{\mathbb{R}}c\l(\frac{x-y}{\epsilon}\r)\Big(u^2(y)-u^2(x)\Big)dydx\\
    &=-\frac{1}{2\epsilon^3}\int_{\mathbb{R}}\int_{\mathbb{R}}\lambda^m\l(\frac{x}{\epsilon}\r)
    c\l(\frac{x-y}{\epsilon}\r)\Big(u(y)-u(x)\Big)^2dydx\leq 0,
     \end{split}
    \end{equation}
    For convenience,  we denote $a^{m,\epsilon}(x)=a^m\l(\frac{x}{\epsilon}\r)$, $u^\epsilon(t,x)=u^\epsilon_t$, for $t\in [0,T]$.
    Integrating both sides of (\ref{5}) with respect to $x$,  we have
    \begin{equation}\label{6}
    \begin{split}
    (m^\epsilon u^\epsilon_t,u^\epsilon_t)&=(m^\epsilon u_0,u_0)-2\int^t_0\Big(a^{m,\epsilon}(u^\epsilon_s)',(u^\epsilon_s)'\Big)ds\\
    &\quad+\int_0^t(m^\epsilon u^\epsilon_s\sigma^\epsilon,u^\epsilon_s\sigma^\epsilon)ds
    +2\int_0^t(m^\epsilon u^\epsilon_s\sigma^\epsilon, u^\epsilon_s)dW_s\\
    &\quad-\frac{1}{2\epsilon^3}\int_0^t\int_{\mathbb{R}}\int_{\mathbb{R}}\lambda^m\l(\frac{x}{\epsilon}\r)
    c\l(\frac{x-y}{\epsilon}\r)\Big(u(y)-u(x)\Big)^2dydxds,
    \end{split}
    \end{equation}
    where $m$ satisfies $\delta<m<\delta^{-1}$ for some $\delta>0$(p.$380$ of \cite{BAL} and \cite{HuangQiao}), and $\sigma(\cdot)$ is bounded. From now on, we will denote by $C_i, i=1,2, \cdots, $ the constants which may depend on $\kappa_1,\delta,T$. Then we have
    \begin{equation}
    \begin{split}\label{7}
    \delta\left\|u^\epsilon_t\right\|^2_0&+2\kappa_1\delta\int_0^t\left\|\nabla u_s^\epsilon\right\|^2_0ds\leq\delta^{-1}\left\|u_0\right\|^2_0+C_2\int_0^t\left\|u_s^\epsilon\right\|^2_0ds\\
    &+2\int_0^t(m^\epsilon u^\epsilon_s\sigma^\epsilon, u^\epsilon_s)dW_s.
    \end{split}
    \end{equation}
    By Gronwall's inequality, we obtain
    $$\sup_\epsilon \sup_{0\leq t\leq T} \mathbb{E}\left[\left\|u^\epsilon_t\right\|^2_0\right]\leq C_3\left(1+\left\|u_0\right\|^2_0\right).$$
    Applying Ito's formula to equation (\ref{6}) we can see
    \begin{equation*}
      \begin{split}
       \left(m^\epsilon u^\epsilon_t,u^\epsilon_t\right)^2&+4\int^t_0\Big(a^{m,\epsilon}(u^\epsilon_s)',(u^\epsilon_s)'\Big)(m^\epsilon u_s^\epsilon,u_s^\epsilon)ds\\
       &\leq(m^\epsilon u_0,u_0)^2
    +2\int_0^t(m^\epsilon u^\epsilon_s\sigma^\epsilon,u^\epsilon_s\sigma^\epsilon)(m^\epsilon u_s^\epsilon,u_s^\epsilon)ds\\
    &+4\int_0^t(m^\epsilon u^\epsilon_s\sigma^\epsilon, u^\epsilon_s)(m^\epsilon u_s^\epsilon,u_s^\epsilon)dW_s
    +4\int_0^t(m^\epsilon u^\epsilon_s\sigma^\epsilon, u^\epsilon_s)^2ds.
      \end{split}
    \end{equation*}
    Now we consider the expectation after taking the supremum with respect to $t$ of both sides, we deduce that
    \[
    \begin{split}
    &C_4\mathbb{E}\left[\sup_\epsilon \sup_{0\leq s\leq t}\left\|u^\epsilon_s\right\|^4_0\right]+C_4\mathbb{E}\left[\int_0^t\left\|\nabla u_s^\epsilon\right\|^2_0\left\|u_s^\epsilon\right\|^2_0ds\right]\\
    &\quad \leq \left\|u_0\right\|^4_0+\int_0^t \mathbb{E}\left[\left\|u^\epsilon_s\right\|^2_0\right]ds+\int_0^t \mathbb{E}\left[\sup_{0\leq r\leq s}\left\|u^\epsilon_r\right\|^4_0\right]ds\\
    &\quad +\mathbb{E}\Bigg[\sup_{0\leq s\leq t}\left|\int_0^s (m^\epsilon u^\epsilon_r,u^\epsilon_r)^2dW_r\right|\Bigg].
    \end{split}
    \]
    By Burkholder-Davis-Gundy inequality \cite{15}, the third term of the right hand side can be estimated as
    \[
    \begin{split}
    &\mathbb{E}\Bigg[\sup_{0\leq s\leq t}\left|\int_0^s \sigma^\epsilon (m^\epsilon u^\epsilon_r,u^\epsilon_r)^2W_r\right|\Bigg]\\
    &\quad \leq C_5\mathbb{E}\Bigg[\Big(\int_0^t(1+\left\|u_s^\epsilon\right\|^2_0)\left\|u_s^\epsilon\right\|_0^6ds\Big)^{\frac{1}{2}}\Bigg]\\
    &\quad \leq \frac{C_5}{2}\mathbb{E}\Bigg[\sup_{0\leq s\leq t}\left\|u_s^\epsilon\right\|^4_0\Bigg]+C_6\mathbb{E}\Bigg[\int_0^t(1+\left\|u_s^\epsilon\right\|^2_0)\left\|u_s^\epsilon\right\|^2_0ds\Bigg].
    \end{split}
    \]
    Thus, we conclude that
    $$\sup_\epsilon\sup_{0\leq t\leq T} \mathbb{E}\left[\left\|u^\epsilon_t\right\|^4_0\right]\leq C_7\left(1+\left\|u_0\right\|^4_0\right).$$
    Moreover, from (\ref{7})
    \[
    \begin{split}
    \mathbb{E}\Bigg[\left(\int_0^T\left\|\nabla u_s^\epsilon\right\|^2_0ds\right)^2\Bigg]&\leq C_8\left(1+\left\|u_0\right\|^4_0\right)+C_8\mathbb{E}\Bigg[\left(\int_0^T\left\|u_s^\epsilon\right\|^2_0ds\right)^2\Bigg]\\
    &\quad +C_8\mathbb{E}\left[\left(\int_0^T\left(m^\epsilon\sigma^\epsilon u_s^\epsilon,u_s^\epsilon\right)dW_s\right)^2\right]\\
    &\leq C_9\left(1+\left\|u_0\right\|^4_0\right).
    \end{split}
    \]
    We can further verify that (\ref{88})    holds for every $u_0\in H$,  by a density argument.
      \end{proof}
    \end{lemma}

    Next, we shall show the equicontinuity of $\{(m^\epsilon u^\epsilon,\varphi)\}_{\epsilon>0}$ for each $\varphi\in C_c^\infty(\mathbb{R}).$

    \begin{lemma}\label{2.4}
    Let $u_t^\epsilon$  be the  solution of  the heterogeneous equation (\ref{1}) with initial value $u_0\in H$. Then, for every $\varphi\in C_c^\infty(\mathbb{R})$, there exists
    a positive constant $C$ such that
    $$\sup_\epsilon \mathbb{E}\l[\l(m^\epsilon u_t^\epsilon-m^\epsilon u_s^\epsilon,\varphi\r)^4\r]\leq C\l|t-s\r|^2\l(1+\l\|u_0\r\|^4_0\r),$$
    for all $s,t\in [0,T]$.
    \begin{proof}
    From the definition of $u^\epsilon$, we infer that
    $$\l(m^\epsilon u^\epsilon_t-m^\epsilon u_s^\epsilon,\varphi\r)^4\leq 8\l(\int_s^t\l(T_m^\epsilon u^\epsilon_r,\varphi\r)dr\r)^4+
    8\l(\int_s^t\l(m^\epsilon u_r^\epsilon \sigma\l(\frac{x}{\epsilon}\r),\varphi\r)dW_r\r)^4.$$
    Note that there exists a  function $\gamma$ in $H^1(\mathbb{T})$, such that
    $$\beta^m(\eta)=\gamma'(\eta), \eta\in\mathbb{T}.$$
    Consider the following partial differential equation on $\mathbb{T}$:
    \[
    \begin{cases}
      (\widetilde{A}_m)^*\varrho-\beta^m=0,\\
      \int_0^1\varrho(\eta)d\eta=0,
    \end{cases}
    \]
    where $\tilde{A}_m=m\tilde{A}.$
    Then $$\gamma(\eta)=a^m(\eta)\varrho'(\eta)-\beta^m\varrho(\eta).$$
    We denote $\l(\Gamma u\r)(x)=\int_{\mathbb{R}}c\l(\frac{x-y}{\epsilon}\r)\l(u(y)-u(x)\r)dy$ and $\beta^{m,\epsilon}=\beta^m\l(\frac{x}{\epsilon}\r)$,$\gamma^\epsilon=\gamma\l(\frac{x}{\epsilon}\r)$. It follows that $\Gamma$ is  a symmetric operator on $L^2(\mathbb{R}).$
    Moreover,
    \[
    \begin{split}
    \left \langle T^\epsilon_mu,\varphi \right \rangle=&\l(a^{m,\epsilon}u',\varphi'\r)-\epsilon^{-1}\l(\beta^{m,\epsilon}u,\varphi'\r)\\
    &+\frac{1}{2}\int_{\mathbb{R}}\lambda^{m,\varepsilon}(x)\int_{\mathbb{R}}c(\frac{x-y}{\epsilon})\bigg[\Big(u(x)\big(\varphi(y)-\varphi(x) \big)\Big)\\
    &+\varphi(y)\big(u(x)-u(y)\big)\bigg]dydx\\
\leq&\l(a^{m,\epsilon}u',\varphi'\r)-\epsilon^{-1}\l(\beta^{m,\epsilon}u,\varphi'\r)+\frac{C_2}{\delta}\big(\Gamma u(x),\varphi(x)\big),
    \end{split}
    \]
    where $\l\|\Gamma u\r\|^2_0\leq a_1^2\l\|u\r\|^2_0.$  Hence
    \[
    \begin{split}
    \left \langle T^\epsilon_mu,\varphi\right \rangle&\leq\l(a^{m,\epsilon}u',\varphi'\r)+(\gamma^\epsilon u',\varphi')+(\gamma^\epsilon u,\varphi'')+C_0||u||_0||\varphi||_0\\
    &\leq C_1||\varphi||_{H^2}||u||_1,
    \end{split}
    \]
    for some $C_0, C_1>0$. Then we deduce that
    \[
    \begin{split}
    \mathbb{E}\l[\l(m^\epsilon u^\epsilon_t-m^\epsilon u_s^\epsilon,\varphi\r)^4\r]&\leq C_2\l\|\varphi\r\|_{H^2}^4\l|t-s\r|^2\Bigg\{\mathbb{E}\l[\l(\int_s^t\l\|u_r^\epsilon\r\|_1^2dr\r)^2\r]
    +1\\&+\mathbb{E}\l[\sup_{0\leq t\leq T}\l\|u_r^\epsilon\r\|^4_0\r]\Bigg\}
    \leq C_3||\varphi||_{H^2}^4\l|t-s\r|^2\l(1+\l\|u_0\r\|^4_0\r).
    \end{split}
    \]
    The proof is complete.
    \end{proof}
    \end{lemma}
As a result of Lemma \ref{2.3} and Lemma \ref{2.4}, we will show the Lemma \ref{2.5}.

      \begin{proof}
     It is obvious that $$\sup_\epsilon \mathbb{E}\l[\sup_{0\leq t\leq T}\l\|m^\epsilon u^\epsilon_t\r\|^4_{L^2}\r]<\infty.$$
     By Kolmogorov's tightness criterion \cite{Bou}, we can obtain the tightness of real valued processes $\{\l(m^\epsilon u^\epsilon,\varphi\r);\epsilon>0\}$ for
     every $\varphi\in C_c^\infty(\mathbb{R})$(Section $4$ of \cite{31}). Since the injection $H\hookrightarrow H_{-\kappa}^{-1}$ is compact,  $\{\pi^{m,\epsilon};\epsilon>0\}$ is tight in $(K,\mathcal{T}_1)$ .
      \end{proof}

  From the tightness for $\{\pi^{m,\epsilon}: \epsilon>0\}$, we conclude with the following result.
    \begin{lemma}\label{2.6}
      There exists a subsequence $\epsilon_k \rightarrow 0$ as $k\rightarrow\infty$, and a probability measure
      $ \widetilde {\pi}$ on $K$ such that
      $$\pi^{m,\epsilon_k}\Rightarrow\widetilde{\pi} \quad \text{in} \quad (K,\mathcal{T}_1).$$
    \end{lemma}
    \begin{remark}
      If we show the uniqueness of $\widetilde{\pi}$ for every subsequence in Lemma \ref{2.6}, then the whole sequence $\{\pi^{m,\epsilon}: \epsilon>0\}$ converges to $\widetilde{\pi}$
       in $(K,\mathcal{T}_1).$
    \end{remark}

 Next, we will show that $\widetilde{\pi}$ is also the limit measure of $\pi^{\epsilon_k}$ in $(K,\mathcal{T}_2).$
    \begin{lemma}\label{2.7}
     Let $\epsilon_k$ be the subsequence in Lemma \ref{2.6} and $\pi^{\epsilon_k}$ be the probability measure induced by $u^{\epsilon_k}$. Then
      $$\pi^{\epsilon_k}\Rightarrow\widetilde {\pi} \quad in \quad (K,\mathcal{T}_2),$$ as $k\rightarrow\infty$.
      \begin{proof}
      From Skorokhod's theorem, there exists a probability space $\l(\widetilde{\Omega},\widetilde{\mathcal{F}},\\ \widetilde{\mathbb{P}}\r)$ with \emph{K}-valued random
      variables $\widetilde{u}^{m,\epsilon_k},\widetilde{u}$ on $\l(\widetilde{\Omega},\widetilde{\mathcal{F}},\widetilde{\mathbb{P}}\r)$, such that $\pi^{m,\epsilon_k}$ and $\widetilde{\pi} $ are the laws of $\widetilde{u}^{m,\epsilon_k}$ and $\widetilde{u}$, respectively. Moreover,
      $$\widetilde{u}^{m,\epsilon_k} \rightarrow \widetilde{u}$$  in $(K,\mathcal{T}_1)$, $\widetilde{\mathbb{P}}$-almost surely.

Now we want to show that
      $$\widetilde{\mathbb{E}}\l[\int_0^T\l\|(m^{\epsilon_k})^{-1}\widetilde{u}_t^{m,\epsilon_k}-\widetilde{u}_t\r\|^2_{H_{-\kappa}}dt\r]\rightarrow 0,$$
      as $k\rightarrow\infty$,where $\widetilde{\mathbb{E}}$ stands for the expectation  with respect to $\widetilde{\mathbb{P}}$.
      First, we have
      \[
      \begin{split}
      &\widetilde{\mathbb{E}}\l[\int_0^T\l\|(m^{\epsilon_k})^{-1}\widetilde{u}_t^{m,\epsilon_k}-\widetilde{u}_t\r\|^2_{H^{-1}_{-\kappa}}dt\r]\\&\leq
     2\widetilde{\mathbb{E}}\l[\int_0^T\l\|(m^{\epsilon_k})^{-1}\widetilde{u}_t^{m,\epsilon_k}(1-m^{\epsilon_k})\r\|^2_{H^{-1}_{-\kappa}}dr\r]\\
      &\quad+2\widetilde{\mathbb{E}}\l[\int_0^T\l\|\widetilde{u}_t^{m,\epsilon_k}-\widetilde{u}_t\r\|^2_{H^{-1}_{-\kappa}}dt\r]\\
      &=2\mathbb{E}\l[\int_0^T\l\|u_t^{\epsilon_k}(1-m^{\epsilon_k})\r\|^2_{H^{-1}_{-\kappa}}dt\r]+2\widetilde{\mathbb{E}}\l[\int_0^T\l\|\widetilde{u}_t^{m,\epsilon_k}-\widetilde{u}_t\r\|
      ^2_{H^{-1}_{-\kappa}}dt\r].
      \end{split}
      \]
      Let $\zeta(\eta)$ be a solution of the  following partial differential equation:
      \[
      \begin{cases}
        \Delta \zeta(\eta)=1-m(\eta),  \eta\in\mathbb{T}\\
        \int_0^1\zeta(\eta)d\eta=0.
      \end{cases}
      \]
      Then, we have
      $$\l\|u_t^{\epsilon_k}(1-m^{\epsilon_k})\r\|_{H^{-1}_{-\kappa}}=\l\|u_t^{\epsilon_k}\Delta\zeta^{\epsilon_k}\r\|_{H^{-1}_{-\kappa}}
      =\sup_{\upsilon\neq 0}\frac{\l|\left \langle u_t^{\epsilon_k}\Delta\zeta^{\epsilon_k},\upsilon\right \rangle\r|}{\l\|\upsilon\r\|_{H_\kappa^1}}
      \leq C\epsilon_k\l\|u_t^{\epsilon_k}\r\|_1,$$
      for some positive constant $C$, independent of $\epsilon_k,$ where
      $\Delta\zeta^{\epsilon}=(\Delta\zeta)(x/\epsilon).$ Therefore, we conclude that
      \[
      \begin{split}
      \widetilde{\mathbb{E}}\l[\int_0^T\l\|(m^{\epsilon_k})^{-1}\widetilde{u}_t^{m,\epsilon_k}-\widetilde{u}_t\r\|^2_{H^{-1}_{-\kappa}}dt\r]&\leq 2\epsilon_k^2C^2\mathbb{E}
      \l[\int_0^T\l\|u_t^{\epsilon_k}\r\|^2_1dt\r]\\
      &\quad+2\widetilde{\mathbb{E}}\l[\int_0^T\l\|\widetilde{u}_t^{m,\epsilon_k}-\widetilde{u}_t\r\|^2_{H^{-1}_{-\kappa}}dt\r]\rightarrow 0,
       \end{split}
      \]
      as $k\rightarrow\infty$.
      Since $H^1, H_{-\kappa},H_{-\kappa}^{-1}$ are reflective Banach space, and the inclusion $H^1\hookrightarrow H_{-\kappa}$ is compact, $H_{-\kappa}\hookrightarrow H_{-\kappa}^{-1}$ is continuous. We konw that, for every $\rho>0$, there exists a constant $C(\rho)>0$ such that
      $$\l\|\upsilon\r\|_{H_{-\kappa}}\leq \rho \l\|\upsilon\r\|_1+C(\rho)\l\|\upsilon\r\|_{H_{-\kappa}^{-1}},\quad \upsilon\in H^1.$$
      Hence the conclusion in this lemma follows.
      \end{proof}
    \end{lemma}

 \subsection{Identification of the Limit Law}\label{ss42}
 In this section, we will verify that $\widetilde{\pi}$ coincides with the law induced by the solution of the effective equation (\ref{2}).

 Let $X=(X_t)$ be the canonical process, that is, $X_t(\omega)=\omega_t$ for $\omega\in K$, and $\mathcal{D}_t$ the canonical filtration on $K$. We define a $\sigma$-field $\mathcal{D}=\bigvee_{0\leq t\leq T }\mathcal{D}_t$.

 \begin{definition}
   A probability measure $\mu$ on $(K,\mathcal{D})$ is called a solution of martingale problem for the effective equation (\ref{2}) if $\mu$ satisfies the following   conditions: \\
    ($\romannumeral 1$)The probability$  \;\mu(X_0(\cdot)=u_0)=1; $ and\\
    ($\romannumeral 2$)For every $\phi\in C_c^\infty(\mathbb{R})$ and $\xi\in C_c^\infty(\mathbb{R})$, the stochastic process $H_{\phi,\xi}(t)$ defined by
   \[
   \begin{split}
   H_{\phi,\xi}(t)=H_{\phi,\xi}^{T^0,M^0}(t)&=\phi\big(\l\langle X_t,\xi\r\rangle \big)-\phi\big(\l\langle X_0,\xi\r\rangle \big)\!-\!\int_0^t \phi'\big(\l\langle X_s,\xi\r\rangle \big)\l\langle X_s,(T^0)^*\xi\r\rangle ds\\
   &\quad-\frac{1}{2}\int_0^1\phi''\big(\l\langle X_s,\xi\r\rangle \big)\big(M^0(X_s),\xi\big)^2ds,
   \end{split}
   \]
     is a continuous local martingale under $\mu$. That is $\mathbb{E}\l(H_{\phi,\xi}(t)|\mathcal{D}_s\r)=H_{\phi,\xi}(s), \\ 0\leq s\leq t\leq T.$
 \end{definition}

 \begin{lemma}\label{2.8}
   The effective equation (\ref{2}) has at most one martingale solution on $(K,\mathcal{D})$.
   \begin{proof}
   It is sufficient to check the pathwise  uniqueness property of the equation  (\ref{2}) on $K$ (p.$89$ of \cite{MMM}). For a   solution $u_t$   of the effective equation (\ref{2}),  we set $\widetilde{u}_t=u_t\emph{e}^{-\kappa\theta}$. Then, $\widetilde{u}$ satisfies
   the equality
   $$(\widetilde{u}_t,\xi)=(u_0\emph{e}^{-\kappa\theta},\xi)+\int_0^t\left\langle\mathcal{I}
   \widetilde{u}_s, \xi\right\rangle ds+C\int_0^t
   (\widetilde{u}_s,\xi)dW_s,\; \forall \xi\in H^1,$$
   where $\mathcal{I}$ is the operator which satisfy the coerciveness condition of the form
   $$-2\left\langle\mathcal{I}\upsilon,\upsilon\right\rangle+\nu_1|\upsilon|_0^2+\nu_2\geq\nu_3
   |\upsilon|_1^2, \; \forall \upsilon\in H^1.$$

  Therefore, similar to Lemma \ref{2.2}, we have
  \begin{center}
   $\mathbb{P}\left(\widetilde{u}_t=\widetilde{\upsilon}_t, \text{in}\; H^{-1},\; \forall t\in[0,T]\right)=1,$
  \end{center}
    for two stochastic processes $\widetilde{u},\widetilde{\upsilon}$ satisfying the effective equation (\ref{2}). Hence we obtain the pathwise uniqueness of stochastic partial differential equation (\ref{2}) on $K$.
   \end{proof}
 \end{lemma}

   We now show that $\widetilde {\pi}$ is the martingale solution for the effective equation (\ref{2}). We give a equivalent condition
of a continuous $\mathcal{D}_t-$martingale under the probability measure $\widetilde {\pi}.$

\begin{lemma}\label{3542}
   The stochastic process $H_{\phi,\xi}(t)$ is a continuous $\mathcal{D}_t-$martingale under the probability measure $\widetilde {\pi}$, that is,
   for every bounded, $\mathcal{D}_s-$measurable function $\Phi$ on $K$, we have
   $$\mathbb{E}^{\widetilde {\pi}}[\Phi\{H_{\phi,\xi}(t)-H_{\phi,\xi}(s)\}]=0,\;0\leq s\leq t\leq T.$$
\begin{proof}
   For simplicity of description, we shall use $\epsilon$ in place of $\epsilon_k$, the subsequence in Lemma \ref{2.6}. Without lost of generality, we assume that $\Phi$ is continuous with respect to the supremum topology of $\mathcal{T}_1$ and $\mathcal{T}_2$, which generates on $K$ the same Borel $\sigma$-algebra as that generated by $\mathcal{T}_1$ or $\mathcal{T}_2.$

   Define $H_{\phi,\xi}^\epsilon(t)$  on $(K,\mathcal{D})$ by
    \[
   \begin{split}
   H_{\phi,\xi}^\epsilon(t)&=\phi(\l\langle m^\epsilon X_t,\xi\r\rangle)-\phi(\l\langle m^\epsilon X_0,\xi\r\rangle)-\int_0^t \phi'(\l\langle m^\epsilon X_s,\xi\r\rangle)\l\langle X_s,(T^0)^*\xi\r\rangle ds\\
   &\quad-\frac{1}{2}\int_0^1\phi''(\l\langle m^\epsilon X_s,\xi\r\rangle)(M^0(X_s),\xi)^2ds.
   \end{split}
   \]
    From Lemma \ref{2.7}, we can extract a subsequence $\{\epsilon_l\}_{l\geq 1}$ such that
    $$(m^{\epsilon_l})^{-1}\widetilde{u}^{m,\epsilon_l} \rightarrow \widetilde{u} \quad in \quad  (K,\mathcal{T}_2),\; \widetilde {\mathbb{P}}-almost\;surely$$
    as $l\rightarrow\infty$. Therefore, we have
\[
\begin{split}
    \mathbb{E}^{\pi^{\epsilon_l}}\l[\Phi\l\{H_{\phi,\xi}^{\epsilon_l}(t)-H_{\phi,\xi}^{\epsilon_l}
    (s)\r\}\r]&=\widetilde{\mathbb{E}}\Big[\Phi\big((m^{\epsilon_l})^{-1}\widetilde{u}^{m,\epsilon_l}\big)\Big\{H_{\phi,\xi}
    ^{\epsilon_l}(t)\big((m^{\epsilon_l})^{-1}\widetilde{u}^{m,\epsilon_l}\big)\\
    &\quad-H_{\phi,\xi}^{\epsilon_l}(s)\big((m^{\epsilon_l})^{-1}\widetilde{u}^{m,\epsilon_l}\big)\Big\}\Big]\\
    &\rightarrow \widetilde{\mathbb{E}}
    \l[\Phi(\widetilde{u})\l\{H_{\phi,\xi}^{\epsilon_l}(t)(\widetilde{u})-H_{\phi,\xi}^{\epsilon_l}
    (s)(\widetilde{u})\r\}\r]\\&= \mathbb{E}^{\widetilde {\pi}}\l[\Phi\l\{H_{\phi,\xi}(t)-H_{\phi,\xi}(s)\r\}\r]
\end{split}
\]
as $l\rightarrow \infty$ in view of bounded convergence theorem. Then we just need to show that
$$\mathbb{E}^{\pi^{\epsilon_l}}\l[\Phi\l\{H_{\phi,\xi}^{\epsilon_l}(t)-H_{\phi,\xi}^{\epsilon_l}
    (s)\r\}\r]\rightarrow 0$$
as $\epsilon\rightarrow 0$ without extracting subsequence. For $\xi_\epsilon \in C_c^\infty(\mathbb{R})$, we define another function $H_{\phi,\xi_\epsilon}^{A^\epsilon,M^\epsilon}(t)$ on $K$ by
\[
   \begin{split}
   H_{\phi,\xi_\epsilon}^{T^\epsilon,M^\epsilon}(t)&=\phi(\l\langle X_t,\xi_\epsilon\r\rangle)-\phi(\l\langle X_0,\xi_\epsilon\r\rangle)-\int_0^t \phi'(\l\langle X_s,\xi_\epsilon\r\rangle)\l\langle X_s,(T^\epsilon)^*\xi_\epsilon\r\rangle ds\\
   &\quad-\frac{1}{2}\int_0^1\phi''(\l\langle X_s,\xi_\epsilon\r\rangle )(\sigma^\epsilon(x)X_s,\xi_\epsilon)^2ds.
   \end{split}
   \]
   It follows that
   $$\mathbb{E}^{\pi^{\epsilon}}\l[\Phi\l\{H_{\phi,\xi}^{T^\epsilon,M^\epsilon}(t)-H_{\phi,\xi}^{T^
   \epsilon,M^\epsilon}
    (s)\r\}\r]=0,\quad 0\leq s\leq t\leq T.$$ So,
   we consider $$\mathbb{E}^{\pi^\epsilon}\l[\Phi\l\{H_{\phi,\xi}^\epsilon(t)-H_{\phi,\xi}^\epsilon(s)-H_{\phi,\xi_\epsilon}^{T^\epsilon,M^\epsilon}(t)
   +H_{\phi,\xi_\epsilon}^{T^\epsilon,M^\epsilon}(s)\r\}\r],$$
   which is equal to
   \begin{equation}\label{8}
     \mathbb{E}\l[\Phi(u^\epsilon)\l\{I_1(u^\epsilon)-I_2(u^\epsilon)-I_3(u^\epsilon)-I_4(u^\epsilon)-I_5(u^\epsilon)\r\}\r],
   \end{equation}
   where
   \[
\begin{split}
   I_1&=\phi(\l\langle m^\epsilon X_t,\xi\r\rangle )-\phi(\l\langle X_t,\xi_\epsilon\r\rangle)-\phi(\l\langle m^\epsilon X_s,\xi\r\rangle)+
   \phi(\l\langle X_s,\xi_\epsilon\r\rangle),\\
   I_2&=\int_s^t\Big\{\phi'(\l\langle m^\epsilon X_r,\xi\r\rangle )-\phi'(\l\langle X_r,\xi_\epsilon\r\rangle )\Big\}\l\langle X_r,(T^0)^*\xi\r\rangle dr,\\
   I_3&=\int_s^t\phi'(\l\langle X_r,\xi_\epsilon\r\rangle )\l\langle X_r,(T^0)^*\xi-(T^\epsilon)^*\xi_\epsilon)\r\rangle dr,\\
   I_4&=\frac{1}{2}\int_s^t\Big\{\phi''(\l\langle m^\epsilon X_r,\xi\r\rangle )-\phi''(\l\langle X_r,\xi_\epsilon\r\rangle )\Big\}
   (M^0(X_r),\xi)^2dr,\\
   I_5&=\frac{1}{2}\int_s^t\phi''(\l\langle X_r,\xi_\epsilon\r\rangle )\Big\{(M^0(X_r),\xi)^2-(\sigma^\epsilon(x)
   X_r,\xi_\epsilon)^2\Big\}dr.
\end{split}
\]

We will construct a family of test functions $\xi^{\epsilon}\in C_c^\infty(\mathbb{R})$ such that (\ref{8}) goes to $0$ as $\epsilon\rightarrow 0$. So we define $\xi^{\epsilon}\in C_c^\infty(\mathbb{R})$ as follows:
 \begin{equation}\label{15}
 \xi^{\epsilon}(x)=m\l(\frac{x}{\epsilon}\r)\l(\xi(x)+\epsilon h_{1}\l(\frac{x}{\epsilon}\r)\xi'(x)+\epsilon^2 h_2 \l(\frac{x}{\epsilon}\r)\xi''(x)\r),
 \end{equation}
We denote $\tilde{T}_m=m \tilde{T}$. The periodic functions $h_{1}, h_{2}$ are, respectively, the unique solutions  of the following equations
 \begin{equation}\label{12}
  \begin{cases}
  (\tilde{T}_m)^*(h_1)(\eta)=l(\eta),\\
  \int_{\mathbb{T}}h_1(\eta)d\eta=0,
\end{cases}
\end{equation}
 \begin{equation}
  \begin{cases}
(\tilde{T}_m)^*(h_2)(\eta)=l_1(\eta),\\
 \int_{\mathbb{T}}h_2(\eta)d\eta=0,
\end{cases}
\end{equation}
where
$$l(\eta)=\int_{\mathbb{R}}zc(z)m(\eta-z)\lambda(\eta-z)dz+b(\eta)m(\eta)-2(a(\eta)m(\eta))',$$
\[
\begin{split}
l_1(\eta)&=-Q+\int_{\mathbb{R}}c(z)\lambda(\eta-z)m(\eta-z)\l[\frac{1}{2}z^2-zh_1(\eta-z)\r]dz\\
&\quad+a(\eta)m(\eta)+2(a(\eta)m(\eta)h_1(\eta))'-b(\eta)m(\eta)h_1(\eta).
\end{split}
\]
\end{proof}
\end{lemma}

The constant $Q$ is defined in (\ref{14}). The existence and uniqueness of the solution $h_1(\eta),h_2(\eta)$ will be given in Appendix B. Then we have the unique form of the test function $\xi_\epsilon.$
In order to prove the formula (\ref{8}) tends to zero, we give the following lemma which will be proved in the Appendix C.
\begin{lemma}\label{2.10}
  Assume that $f\in \mathcal{S}(\mathbb{R})$, which is the space of rapidly decaying functions.  Then there exist functions $h_1,h_2\in L^2(\mathbb{T})$
  and a positive constant $Q$ defined in (\ref{14}),  such that for the function $\xi^\epsilon$
  defined by (\ref{15}),  we have
  $$(T^\epsilon)^*(\xi^\epsilon)=Q \xi''+\phi_\epsilon,$$ where $$\lim_{\epsilon\rightarrow0}
  \Vert\phi_\epsilon\Vert_0=0.$$
\end{lemma}

At last, we can show that $\mathbb{E}\l[\r|I_i(u^\epsilon)\l|\r]\rightarrow 0,$ as $\epsilon$ goes to $0.$ That is to say $\widetilde {\pi}$ is the martingale solution for the effective equation (\ref{2}).

\begin{lemma}\label{2.11}
   The stochastic process $H_{\phi,\xi}(t)$ is a continuous $\mathcal{D}_t-$martingale under the probability measure $\widetilde {\pi}$. That is,
   for every bounded, $\mathcal{D}_s-$measurable function $\Phi$ on $K$, we have
   $$\mathbb{E}^{\widetilde {\pi}}[\Phi\{H_{\phi,\xi}(t)-H_{\phi,\xi}(s)\}]=0,\;0\leq s\leq t\leq T.$$
\begin{proof}
 We just need to prove that  $\mathbb{E}[|I_i(u^\epsilon)|]\rightarrow 0$ in (\ref{8}),
  for each $i=1,2,...5.$  From
$$\l|(m^\epsilon u_t^\epsilon,\xi)-(u_t^\epsilon,\xi_\epsilon)\r| \leq ||m^\epsilon u_t^\epsilon||_0||\xi-(m^\epsilon)^{-1}\xi_\epsilon||_0\rightarrow0,$$
we obtain that $\mathbb{E}[|I_1(u^\epsilon)|]\rightarrow 0.$
Similarly, for $i=2, 4,$ we can show that
$\mathbb{E}[|I_i(u^\epsilon)|]\rightarrow 0.$
For $i=3$, we have
$$\mathbb{E}[|I_3(u^\epsilon)|]\rightarrow 0 \leq C\l\|T^0)^*\xi-(T^\epsilon)^*\xi_\epsilon\r\|_{H^{-1}}\mathbb{E}\l[\int_0^T\l\|u_r^\epsilon\r\|_1dr\r]\rightarrow 0.$$
Finally, we will show $\mathbb{E}[|I_5(u^\epsilon)|]\rightarrow 0$ as $\epsilon\rightarrow 0.$
Since $u^\epsilon$ and $\l(m^\epsilon\r)^{-1}\widetilde{u}^{m,\epsilon}$ have the same law on $K$, we deduce that
\[
\begin{split}
&\quad\l(\l(m^\epsilon\r)^{-1}\widetilde{u}_r^{m,\epsilon}\sigma^\epsilon(x),\xi_\epsilon\r)-\l(M^0\l(\l(m^\epsilon\r)^{-1}\widetilde{u}_r^{m,\epsilon}\r),\xi\r)\\
&=\l(\l(m^\epsilon\r)^{-1}\widetilde{u}_r\sigma^\epsilon(x)-\widetilde{u}_r^{m,\epsilon}\sigma^\epsilon(x),\xi_\epsilon\r)+\Big(M^0(\widetilde{u}_r)-
M^0\l((m^\epsilon)^{-1}\widetilde{u}_r^{m,\epsilon})\r),\xi\Big)\\
&+\l(m^\epsilon\widetilde{u}_r\sigma^\epsilon(x),(m^\epsilon)^{-1}\xi_
\epsilon-\xi\r)
\l((m^\epsilon\widetilde{u}_r\sigma^\epsilon(x)-M^0(\widetilde{u}_r),\xi\r)\rightarrow 0,
\end{split}
\]
as $\epsilon\rightarrow 0$ by using the convergence
$$\l((m^\epsilon u\sigma^\epsilon(x)-M^0(u),\xi\r)\rightarrow 0,$$
for each $u\in H_{-\kappa}.$
Hence $\mathbb{E}[|I_5(u^\epsilon)|]\rightarrow 0$. The proof  is complete.
\end{proof}
\end{lemma}

\bigskip

 \textbf{Proof of Theorem $\ref{1.1}$}   \\

  Since the uniqueness of the martingale solution for the effective equation (\ref{2}) and the conclusion in  Lemma \ref{2.6}, we know that $\pi^{m,\epsilon}\Rightarrow \pi$ in $\l(K, \mathcal{T}_1\r).$ On the other hand, $\pi^\epsilon,$  the law of the solution of
   the heterogeneous equation (\ref{1}), goes to $\pi$ in $(K, \mathcal{T}_2)$ by Lemma \ref{2.7}. At the same time, by Lemma \ref{2.11}, we know that  $\pi$ is the martingale solution for the effective equation (\ref{2}). In conclusion Theorem \ref{1.1} is proved.


%
%
%
%

\section{Proof of Theorem \ref{27}}\label{s5}

The proof of \textbf{Theorem \ref{27}} is also divided into two steps: the tightness and the limit law.

\subsection{Tightness}
We denote $\pi_1^{m_1,\epsilon}$   the probability measure induced by $m_1^\epsilon v^\epsilon$. We can infer the  next lemma.
\begin{lemma}\label{33}
      If $v^\epsilon$ is a solution of heterogeneous equation (\ref{21}), and $\pi_1^{m_1,\epsilon}$ is the probability measure induced by $m_1^\epsilon v^\epsilon$. Then,
      $\{\pi^{m_1,\epsilon}:\epsilon>0\}$ is tight in $(K_1,\mathcal{T}_1)$.
    \end{lemma}
Thanks to the Lemma \ref{2.3}, we can obtain several uniform estimates concerning the solution $v^\epsilon$ for the original heterogenous system (\ref{22}).
\begin{lemma}\label{30}
    Let $v^\epsilon$\;be a solution of heterogeneous equation (\ref{21}) with initial value $v_0\in H^{\alpha/2}$. Then there exists a positive constant C,  independent of $\epsilon$, such that
    \begin{equation}
    \sup_\epsilon \mathbb{E}\l[\sup_{0\leq t\leq T}\l\|v^\epsilon_t\r\|_0^4\r]+\sup_\epsilon \mathbb{E}\l[\l(\int_0^T\l\|v^\epsilon_t\r\|^2_{H^{\alpha/2}} dt\r)^2\r]\leq C(1+\l\|v_0\r\|^4_0).
    \end{equation}
    \end{lemma}
    Next, we shall show the equicontinuity of $\{(m_1^\epsilon v^\epsilon,\varphi)\}_{\epsilon>0}$ for each $\varphi\in C_c^\infty(\mathbb{R}).$

    \begin{lemma}\label{31}
    Let $v^\epsilon$  be the  solution of  the heterogeneous equation (\ref{21}) with initial value $v_0\in H$. Then, for every $\varphi\in C_c^\infty(\mathbb{R})$, there exists
    a positive constant $C$ such that
    $$\sup_\epsilon \mathbb{E}\l[\l(m_1^\epsilon v_t^\epsilon-m_1^\epsilon v_s^\epsilon,\varphi\r)^4\r]\leq C\l|t-s\r|^2\l\|v_0\r\|^4_0,$$
    for all $s,t\in [0,T]$.
    \begin{proof}
    By the definition of $v^\epsilon,$ we can see
    \begin{equation*}
      \begin{split}
     \l(m_1^\epsilon v^\epsilon_t-m_1^\epsilon v_s^\epsilon,\varphi\r)^4&\leq 27\l(\int_s^t\l\langle F_{m_1}^\epsilon v^\epsilon_r,\varphi\r\rangle dr\r)^4+27\l(\int_s^t\l\langle L_{m_1}^\epsilon v^\epsilon_r,\varphi\r\rangle^4dr\r)\\
     &+27\l(\int_s^t\l(m_1^\epsilon v^\epsilon \sigma^\epsilon,\varphi\r)dW_r\r)^4.
    \end{split}
    \end{equation*}
Then, For the operators $F_{m_1}^\epsilon$ and $L_{m_1}^\epsilon,$ we have $$\l\langle F_{m_1}^\epsilon v^\epsilon_r,\varphi \r\rangle \leq C_0\l\|v^\epsilon_r\r\|_0\l\|\varphi\r\|_1,$$
$$\l\langle L_{m_1}^\epsilon v,v\r\rangle=-\frac{1}{2}\big(\mathcal{D^*}v(x,y),\delta^{\alpha,\epsilon}(x)m_1^\epsilon(x)\mathcal{D^*}v(x,y)\big)_{L^2(\mathbb{R}\times\mathbb{R})}\leq0.$$
We use $\lambda^\epsilon$ to denote the eigenvalues of $L_{m_1}^\epsilon$, from the fact that the function $\delta,m$ are bounded, we know that there exists a positive constant $C_1,C_2,C_3$ such that
$$(\lambda^\epsilon v,v)\geq-C_1\|v\|^2_{H^{\alpha/2}}.$$
Then,
\begin{equation*}
  \begin{split}
   \mathbb{E}\l(\int_s^t\l\langle L_{m_1}^\epsilon v^\epsilon_r,\varphi\r\rangle dr\r)^4&\leq\mathbb{E}\l(\int_s^t-\lambda^\epsilon \l\|v^\epsilon_r\r\|_{H^{\alpha/2}}\l\|\varphi\r\|_0dr\r)^4\\
   &\leq\l(\lambda^\epsilon\r)^4||\varphi||_0^4\mathbb{E}\l(\int_s^t\l\|v^\epsilon_r\r\|_{H^{\alpha/2}}dr\r)^4\\
   &\leq(\lambda^\epsilon)^4C_1^2||\varphi||_0^4\mathbb{E}\l(\int_s^t\l\|v^\epsilon_r\r\|_{H^{\alpha/2}}dr\r)^4\\
   &\leq C_3\l\|\varphi\r\|_0^4\l|t-s\r|^2\l\|v_0\r\|^4_0.
  \end{split}
\end{equation*}
We can conclude that
    $$\sup_\epsilon \mathbb{E}\l[(m_1^\epsilon v_t^\epsilon-m_1^\epsilon v_s^\epsilon,\varphi)^4\r]\leq C\l|t-s\r|^2\l\|v_0\r\|^4_0.$$
    \end{proof}
    \end{lemma}
    We can also infer the next three lemmas in the
    same way we used in Lemma \ref{2.6}, Lemma \ref{2.7} and Lemma \ref{2.8}.
    \begin{lemma}\label{34}
      There exists a subsequence $\epsilon_k \rightarrow 0$ as $k\rightarrow\infty$, and a probability measure
      $ \widetilde {\pi_1}$ on $K_1$ such that
      $$\pi^{m_1,\epsilon_k}\Rightarrow\widetilde{\pi_1} \quad in \quad (K_1,\mathcal{T}_1).$$
    \end{lemma}

    \begin{lemma}\label{35}
     Let $\epsilon_k$ be the subsequence in Lemma \ref{2.6} and $\pi_1^{\epsilon_k}$ be the probability measure induced by $v^{\epsilon_k}$. Then
      $$\pi_1^{\epsilon_k}\Rightarrow\widetilde {\pi_1} \quad in \quad (K_1,\mathcal{T}_2),$$ as $k\rightarrow\infty$.
      \end{lemma}
      \begin{lemma}\label{36}
    The effective equation (\ref{21}) has at most one martingale solution on $(K_1,\mathcal{D})$.
    \end{lemma}

      \subsection{Identification of the Limit Law}
    Recall the calculation in Subsection \ref{ss42}, in order to obtain that $\widetilde {\pi_1}$ is the martingale solution for the effective equation (\ref{21}), we only need to show the convergence of $(V^\epsilon)^*\xi^\epsilon$ as $\epsilon$ goes to $0$.

     We will construct a family of test functions $\xi^{\epsilon}\in C_c^\infty(\mathbb{R})$ as follows:
      \begin{equation}\label{37}
 \xi^{\epsilon}(x)=m_1\l(\frac{x}{\epsilon}\r)\l(\xi(x)+\epsilon h_{3}\l(\frac{x}{\epsilon}\r)\xi'(x)\r),
 \end{equation}
 where the function $h_3$ is the solution of the following equations:
 \begin{equation}\label{38}
  \begin{cases}
   (\widetilde{L_{m_1}})^*(h_3)(\eta)=p(\eta)m_1(\eta),\;\eta\in\mathbb{T},\\
   \int_0^1h_3(\eta)d\eta=1.
 \end{cases}
\end{equation}

Then we have
\[
    \begin{split}
    \l\langle(V^\epsilon)^*\xi^\epsilon,\psi\r\rangle&\rightarrow \bigg(\int_{\mathbb{T}}\delta^\alpha(\eta)m_1(\eta)d\eta\cdot\l(-(-\Delta)^{\alpha/2}\xi(x)\r)+\xi'(x)\int_{\mathbb{T}}g(\eta)m_1(\eta)d\eta\\
    &+\xi(x)\int_{\mathbb{T}}f(\eta)m_1(\eta)d\eta,\psi\bigg),
    \end{split}
     \]
which will be proved in Appendix D.

Moreover, Theorem \ref{27} is proved through the way in Lemma \ref{2.11}.

\section{Application to Data Assimilation}\label{s6}
In this section, the results in Theorem \ref{1.1} and Theorem \ref{27} will be used for the effective reduction for nonlocal Zakai equations of nonlinear data assimilation system with $\a$-stable l\'{e}vy fluctuations.
We can make a transformation on the solution of Zakai equation such that the operator in the new equation satisfies the assumption condition. In this way we get the effective reduction for stochastic differential equation with a new type of operator.
\subsection{Application of Theorem \ref{1.1}}

We consider the following nonlinear data assimilation problem
\begin{equation}\label{16}
\begin{cases}
  dx_t=\frac{1}{\epsilon}b\l(\frac{x}{\epsilon}\r)dt+\sigma_1\l(\frac{x}{\epsilon}\r)dw_t+dL^{\epsilon}_t,\\
  dy_t=\sigma\l(\frac{x}{\epsilon}\r)dt+dW_t,
\end{cases}
\end{equation}
where $x_t$ is the system state (or signal) and $y_t$ is the observation.  Here
$w^{\epsilon}_t, W^{\epsilon}_t$ are  mutually independent one dimensional Brownian motions, and
$L^{\epsilon}_t$ is a L\'{e}vy process with generator
$$(B^\epsilon u)(x)=\frac{1}
  {\epsilon^3}\lambda\l(\frac{x}{\epsilon}\r)\int_{\mathbb{R}}c\l(\frac{x-y}{\eps}\r)(u(y)-u(x))dy.$$
Then the Zakai equation for the conditional probability density function $u$ of  the data assimilation system (\ref{16}) is the following nonlocal stochastic partial differential equation
\begin{equation}\label{111}
\begin{cases}
$$du^\epsilon(t,x)=(T^\epsilon)^* u^\epsilon(t,x)dt+u^\epsilon(t,x)\sigma\l(\frac{x}{\epsilon}\r)^2dt+u^\epsilon(t,x)\sigma\l(\frac{x}{\epsilon}\r)dW_t,\\
u^\epsilon(0,x)=u_0(x).$$
\end{cases}
\end{equation}

\begin{theorem}
  The family of laws $\{\pi^{\epsilon}; \epsilon>0\}$ induced by the solutions of
  equation (\ref{111}) converges weakly in $(K,\mathcal{T}_1)$ to the law of the
  following stochastic differential equation

  \begin{equation}
  \begin{cases}
    du(t,x)=\widehat{T}^0u(t,x)dt+\int_{\mathbb{T}}m(\eta)\sigma(\eta)^2d\eta\cdot u(t,x)dt+M^0u(t,x)dW_t,\\
    u(0,x)=u_0(x),
  \end{cases}
  \end{equation}
  where $M^0$ is the same operator as (\ref{17}), and
  $$(\widehat{T}^0u)(x)=\l(Q+\int_{\mathbb{T}}2P\widehat{\chi}'d\eta\r)u''(x),$$
 where $Q$ is the same constant as (\ref{14}), 
$P(\eta)=\frac{1}{2}\int_{\mathbb{R}}z^2c(z)\l(\lambda^m(\eta-z)-\lambda^m(\eta)\r)dz,$
and $\widehat{\chi}(\eta)$ is the solution of equation (\ref{mmnm}).

 \medskip
\begin{proof}

 In fact,  set $\widehat{u}^\epsilon=(m^\epsilon)^{-1}u^\epsilon$. Then the  system  (\ref{111}) can be rewritten as
    \begin{equation}
    \begin{cases}
      d\widehat{u}^\epsilon(t,x)=\widehat{T}^\epsilon\widehat{u}^\epsilon(t,x)dt+\widehat{u}^\epsilon(t,x)\sigma\l(\frac{x}{\epsilon}\r)^2dt+\sigma\l(\frac{x}{\epsilon}\r)\widehat{u}^\epsilon(t,x)dW_t,\\
      \widehat{u}^\epsilon(0)=(m^\epsilon)^{-1}u_0,
    \end{cases}
    \end{equation}
    where
    \begin{small}
    \[
    \begin{split}
    &(\widehat{T}^\epsilon u)(x)=\frac{1}{m^\epsilon(x)}\l(A_m^\epsilon\r)^*u(x)+\frac{1}{\epsilon^3m^\epsilon(x)}\int_{\mathbb{R}}c\l(\frac{x-y}{\epsilon}\r)
    \l(\lambda^m\l(\frac{y}{\epsilon}\r)u(y)\!-\!\lambda^m(\frac{x}{\epsilon})u(x)\r)dy\\
    &=a\l(\frac{x}{\epsilon}\r)u''-\frac{1}{\epsilon}\l[b\l(\frac{x}{\epsilon}\r)-\frac{2\l(am\r)'\l(\frac{x}{\epsilon}\r)}{m\l(\frac{x}{\epsilon}\r)}
  \r]u'+\frac{1}{\epsilon^2m^\epsilon(x)}\l((am)''(\frac{x}{\epsilon})-(bm)'(\frac{x}{\epsilon})\r)u(x)\\
  &\quad+\frac{1}{\epsilon^3m^\epsilon(x)}\int_{\mathbb{R}}c\l(\frac{x-y}{\epsilon}\r)
\l(\lambda^m\l(\frac{y}{\epsilon}\r)u(y)-\lambda^m\l(\frac{x}{\epsilon}\r)u(x)\r)dy.
  \end{split}
\]
\end{small}

 For the second term of the operator $\widehat{T}^\epsilon,$
     \begin{small}
     \[
    \begin{split}
    &\epsilon^{-3}(m^\epsilon)^{-1}\!\int_{\mathbb{R}}\!c\l(\frac{x-y}{\epsilon}\r)
    \l(\lambda^m\l(\frac{y}{\epsilon}\r)u(y)\!-\!\lambda^m(\frac{x}{\epsilon})u(x)\r)dy\\
    &=\epsilon^{-2}(m^\epsilon)^{-1}\int_{\mathbb{R}}c(z)\l[u(x)\l(\lambda^m\l(\frac{x}{\epsilon}-z\r)-\lambda^m(\frac{x}{\epsilon})\r)
    +\lambda^m\l(\frac{x}{\epsilon}-z\r)(u(y)-u(x))\r]dz\\
    &=\epsilon^{-2}(m^\epsilon)^{-1}\widetilde{B}^*m\cdot u(x)+B^\epsilon u(x)\\
    &\quad+\epsilon^{-2}(m^\epsilon)^{-1}\int_{\mathbb{R}}c(z)
    \l(\lambda^m(\frac{x}{\epsilon}-z)-\lambda^m(\frac{x}{\epsilon})\r)\l(-\epsilon zu'(x)+(\epsilon z)^2u''(x)+o(\epsilon^2)\r)dz.\\
    \end{split}
    \]
    \end{small}

    we conclude that $$(\widehat{T}^\epsilon u)(x)=B^\epsilon u(x)+\widehat{a}^\epsilon(x)u''(x)+\frac{1}{\epsilon}\widehat{b}^\epsilon(x)u'+o(\epsilon),$$
    where
    $$\widehat{a}^\epsilon(x)=a^\epsilon(x)+\frac{P^\epsilon(x)}{m^\epsilon(x)}$$
    $$\widehat{b}^\epsilon(x)=2\frac{(am)'(\frac{x}{\epsilon})}{m^\epsilon(x)}-b^\epsilon(x)-\frac{Q^\epsilon(x)}{m^\epsilon(x)}$$
    $$Q^\epsilon(x)=\int_{\mathbb{R}}zc(z)\lambda^m(\frac{x}{\epsilon}-z)dz$$
    $$P^\epsilon(x)=\frac{1}{2}\int_{\mathbb{R}}z^2c(z)\l(\lambda^m(\frac{x}{\epsilon}-z)-\lambda^m(\frac{x}{\epsilon})\r)dz$$

    It is easy to see that  $\widehat{T}^\epsilon$ satisfy all  Assumptions $(\romannumeral 1)-(\romannumeral 5)$. Therefore, Theorem \ref{1.1} deduce that the family of laws induced by $u^\epsilon=
    m^\epsilon\widehat{u}^\epsilon$ converges weakly in $(K,\mathcal{T}_1)$ to the law $\pi$ induced by the solution of the following effective   Zakai equation:
    \[
    \begin{cases}
      du(t,x)=\widehat{T}^0u(t,x)dt+\int_{\mathbb{T}}m(\eta)\sigma^2(\eta)d\eta\cdot u(t,x)dt+M^0u(t,x)dW_t,\\
      u(0,x)=u_0(x),
    \end{cases}
    \]
    where
    $$\l(\widehat{T}^0u\r)(x)=Q_1u''(x),$$
    \begin{equation*}
    \begin{split}
    Q_1&=\int_{\mathbb{T}}\widehat{a}(\eta)m(\eta)\l(\widehat{\chi}'(\eta)+1\r)^2dy
    +\frac{1}{2}\int_{\mathbb{T}}\int_{\mathbb{R}}c(\eta-q)\lambda(q)m(q)\big[(\eta-q)\\
    &\quad+\l(\widehat{\chi}(\eta)-\widehat{\chi}(q)\r)\big]^2dydq,
   \end{split}
    \end{equation*}
    and $\widehat{\chi}(\eta)$ is the solution of
    \begin{equation}\label{mmnm}
    \begin{cases}
      \widetilde{T}_1\widehat{\chi}(\eta)+\widehat{b}(\eta)=0, \;\eta\in\mathbb{T},\\
      \int_{\mathbb{T}}\widehat{\chi}(\eta)m(\eta)d\eta=0,
    \end{cases}
    \end{equation}
where $\widetilde{T}_1u=\widetilde{B} u(\eta)+\widehat{a}(\eta)u''(\eta)+\widehat{b}(\eta)u'(\eta).$

We note that $\widehat{\chi}$ coincides with the solution of equation (\ref{12}) up to constant.

First, we can show
\begin{equation}
  \begin{split}
    -\int_{\mathbb{T}}&\widehat{b}^m(\eta)\widehat{\chi}(\eta)d\eta=\l(\widetilde{B}^m\widehat{\chi}+\widehat{a}^m\widehat{\chi}''
    +\widehat{b}^m\widehat{\chi}', \widehat{\chi}\r)\\
    &=(\frac{1}{2}m,\widetilde{B}\widehat{\chi}^2)-\frac{1}{2}\int_{\mathbb{T}}\lambda^m(\eta)\int_{\mathbb{R}}c(x-\eta)(\widehat{\chi}(x)-\widehat
    {\chi}(\eta))dxd\eta\\
    &\quad-\int_{\mathbb{T}}\widehat{a}^m\l(\widehat{\chi}'\r)^2d\eta+\frac{1}{2}\int_{\mathbb{T}}\l(\widehat{a}^m\r)''\widehat{\chi}^2d\eta
    -\frac{1}{2}\int_{\mathbb{T}}(\widehat{b}^m)'\widehat{\chi}^2d\eta\\
    &=-\frac{1}{2}\int_{\mathbb{T}}\lambda^m(\eta)\int_{\mathbb{R}}c(x-\eta)(\widehat{\chi}(x)-\widehat
    {\chi}(\eta))dxd\eta-\int_{\mathbb{T}}\widehat{a}^m\l(\widehat{\chi}'\r)^2d\eta,
  \end{split}
\end{equation}
in view of the equality $\widetilde{T}_1^*m(\eta)=0.$ By the fact that $P'(\eta)=\widehat{b}^m(\eta)-b^m(\eta),$ we have
\begin{small}
\begin{equation}
  \begin{split}
  -\int_{\mathbb{T}}&\widehat{b}^m(\eta)\widehat{\chi}(\eta)d\eta=-\frac{1}{2}\int_{\mathbb{T}}\lambda^m(\eta)\int_{\mathbb{R}}c(x-\eta)(\widehat{\chi}(x)-\widehat
    {\chi}(\eta))dxd\eta+\int_{\mathbb{T}}a^md\eta\\
    &-\int_{\mathbb{T}}2P'\widehat{\chi}d\eta-\int_{\mathbb{T}}\widehat{a}^m\l(\widehat{\chi}'+1\r)^2d\eta-\int_{\mathbb{T}}Q(\eta)\widehat{\chi}d\eta.
   \end{split}
\end{equation}
\end{small}
Thus, the relation $Q_1=Q+\int_{\mathbb{T}}2P\widehat{\chi}'d\eta$ follows from the fact that $\widehat{\chi}(\eta)-h_1(\eta)$ is a constant and
$\int_{\mathbb{T}}\widehat{b}(\eta)m(\eta)=0,$ and the equality (\ref{08908}).
\end{proof}
\end{theorem}
\bigskip

\subsection{Application of Theorem \ref{27}}
In this section, we present an  application of Theorem \ref{27}.
We consider the nonlinear data assimilation problem (\ref{nn})
\begin{equation*}
\begin{cases}
  dx_t=\frac{1}{\epsilon}p\l(\frac{x_t}{\epsilon}\r)dt+\delta{\l(\frac{x_t}{\epsilon}\r)}dL^\alpha_t,\\
  dy_t=\sigma\l(\frac{x_t}{\epsilon}\r)dt+dW_t,
\end{cases}
\end{equation*}
then the \textbf{nonlocal Zakai equation} for the conditional probability density function $v$ of  the data assimilation system (\ref{nn}) is the  following :
\begin{small}
\begin{equation}\label{35367}
\begin{cases}
$$dv^\epsilon(t,x)=\l(L^\epsilon\r)^* v^\epsilon(t,x)dt+v^\epsilon(t,x)\sigma\l(\frac{x}{\epsilon}\r)^2dt+v^\epsilon(t,x)\sigma\l(\frac{x}{\epsilon}\r)dW_t,\\
v^\epsilon(0,x)=v_0(x).$$
\end{cases}
\end{equation}
\end{small}

\begin{theorem}
  The family of laws $\{\pi^{\epsilon}_1; \epsilon>0\}$ induced by the solutions of
  equation (\ref{35367}) converges weakly in $(K_1,\mathcal{T}_1)$ to the law of the
  following stochastic differential equation
  \begin{small}
  \begin{equation}
  \begin{cases}
    dv(t,x)=\widehat{L}^0v(t,x)dt+\int_{\mathbb{T}}m_1(\eta)\sigma(\eta)^2d\eta\cdot v(t,x)dt+M^0v(t,x)dW_t,\\
    v(0,x)=v_0(x),
  \end{cases}
  \end{equation}
  \end{small}
  where $M^0$ is the same operator as (\ref{17}), and
  $$(\widehat{L}^0v)(x)=\int_\mathbb{T}\delta^\alpha(\eta)m_1(\eta)d\eta\cdot(-(-\Delta)^{\alpha/2})v(x).$$

\begin{proof}
 In fact,  set $\widehat{v}^\epsilon=(m_1^\epsilon)^{-1}v^\epsilon$. Then the  system  (\ref{nn}) can be rewritten as
    \begin{equation}
    \begin{cases}
      d\widehat{v}^\epsilon(t,x)=\widehat{L}^\epsilon\widehat{v}^\epsilon(t,x)dt+\widehat{v}^\epsilon(t,x)\sigma\l(\frac{x}{\epsilon}\r)^2dt+\sigma\l(\frac{x}{\epsilon}\r)\widehat{v}^\epsilon(t,x)dW_t,\\
      \widehat{v}^\epsilon(0)=(m_1^\epsilon)^{-1}v_0,
    \end{cases}
    \end{equation}
    where
    \begin{small}
    \begin{equation*}
    \begin{split}
    (\widehat{L}^\epsilon v)(x)
    &=-\frac{1}{\epsilon^{\a-1}}p\l(\frac{x}{\epsilon}\r)v'+(m_1^\epsilon)^{-1}\int_{\mathbb{R}}c\l(\frac{x-y}{\epsilon}\r)
    \l[\delta_1^{m_1}\l(\frac{y}{\epsilon}\r)(v(y)-v(x)\r]dy\\
    &=-\delta_1^\epsilon(x)(-\Delta)^{\a/2}v(x)-\frac{1}{\epsilon^{\a-1}}\l(p({\frac{x}{\epsilon}})+(m_1^\epsilon)^{-1}\int_{\mathbb{R}}\delta_1^m(\frac{x}{\epsilon}-z)\frac
    {z}{|z|^{1+\a}}dz\r)v'(x).
    \end{split}
    \end{equation*}
    \end{small}
in view of the equality $\tilde{L}^* m_1(\eta)=0$


    Then $\widehat{L}^\epsilon$ satisfy the Assumptions $(a), (b)$. Using Theorem \ref{27}, we infer that the family of laws induced by $v^\epsilon=
    m_1^\epsilon\widehat{v}^\epsilon$ converges weakly in $(K_1,\mathcal{T}_1)$ to the law $\pi_1$ induced by the solution of the following effective   Zakai equation:
    \[
    \begin{cases}
      dv(t,x)=V^0v(t,x)dt+M^0v(t,x)dW_t,\\
      v(0,x)=v_0(x),
    \end{cases}
    \]
    where
    $$V^0v(t,x)=\widehat{L}^0v(t,x)+\int_{\mathbb{T}}m_1(\eta)\sigma^2(\eta)d\eta\cdot v(t,x).$$
    We can see that the above effective Zakai equation   coincides with the equation (\ref{22}).
    \end{proof}
    \end{theorem}

\begin{acknowledgements}
 We are grateful to  Qiao Huang and Ao Zhang for helpful comments.
\end{acknowledgements}

\begin{appendices}
 \section*{Appendix A: Proof of Lemma \ref{2.1}}
\begin{proof}
  Note that $A^\epsilon$ is the infinitesimal generator of a $C_0$ semigroup $S(t)$ on $H$, as known in Stewart \cite{Eb}.  Moreover,
  \begin{equation}\label{4}
  \begin{split}
  &\int_{0}^{T}\left\|B^\epsilon u^\epsilon_s(\cdot)\right\|_0+\left\|\l(u^\epsilon_s(\cdot)\r)\sigma(\frac
  {\cdot}{\epsilon})\right\|_0^2ds\\
  &\leq C_2\int_{0}^{T}\left\|\int_{\mathbb{R}}c(\frac{\cdot-y}{\epsilon})u^\epsilon_s(y)dy\right\|_0ds\\
  &+ C_2\int_{0}^{T}\left\|u^\epsilon_s(\cdot)\right\|_0ds\left\|\int_{\mathbb{R}}c(\frac{\cdot-y}{\epsilon})dy\right\|_0\\
  &+\int_{0}^{T}\left\|(u^\epsilon_s(\cdot))^2\sigma(\frac
  {\cdot}{\epsilon})^2\right\|_0ds.
  \end{split}
  \end{equation}
 Combined with the uniform estimates in Lemma \ref{2.3}, we conclude that
 \begin{small}
 \begin{equation*}
 \begin{split}
  \left\|\int_{\mathbb{R}}c(\frac{\cdot-y}{\epsilon})u^\epsilon_s(y)dy\right\|_0^2&\leq
  \int_{\mathbb{R}}c(y)dy\int_{\mathbb{R}}c(q)dq\int_{\mathbb{R}}u^\epsilon_s(x+\epsilon y)u^\epsilon_s(x+\epsilon q)dx\\
  &\leq a_1^2\left\|u_s^\epsilon\right\|^2_0<\infty.
 \end{split}
 \end{equation*}
 \end{small}

  So the right hand side of  (\ref{4}) is finite. Hence the   equation (\ref{1}) has a solution given by
  $$u^\epsilon_t(x)\!=\!S(t)u_0(x)\!+\!\int_0^tS(t\!-\!s)B^\epsilon u^\epsilon_s(x)ds\!+\!\int_0^t
  S(t\!-\!s)u^\epsilon_s(x)\sigma\l(\frac{x}{\epsilon}\r)dW_s.$$
  The mild solution of the equation is unique (Theorem $3.5$ of \cite{GL}).
\end{proof}

\section*{Appendix B:   Uniqueness of  $h_{1}(\eta) \;\text{and} \;h_{2}(\eta)$}
\begin{proof}
 We define the bilinear form:
\[
\begin{split}
a[u,v]&=\int_{\mathbb{T}}\l[\int_{\mathbb{R}}c(\eta-q)\big(\lambda^m(q)u(q)-\lambda^m(\eta)u(\eta)\big)
dq\r] v(\eta)d\eta\\
\quad&+\int_{\mathbb{T}}(a^m(\eta)u(\eta))''v(\eta)d\eta-\int_{\mathbb{T}}(b^m(\eta)u(\eta))'v(\eta)d\eta.
\end{split}
\]
for every $u, v\in H^1.$

At first, we verify the conditions of the Fredholm alternative theorem. We want to show that
there exist positive constants $\nu,\mu,$ such that:
  $$\left| a[u,v]\right|\leq\nu\Vert u\Vert_1\Vert v\Vert_1,$$
  and
  $$\frac{\kappa_1}{2}\Vert u\Vert_1^2\leq a[u,u]+\mu \Vert u\Vert_0^2,$$
  for every $u,v\in H^1(\mathbb{T}).$
Note that
\begin{equation}\label{10}
\begin{split}
\l| a[u,v]\r| & \leq\l|\int_{\mathbb{T}}\l[\int_{\mathbb{R}}c(\eta-q)\l(\lambda^m(q)u(q)-\lambda^m(\eta)u(\eta)\r)
dq \r]v(\eta)d\eta\r|\\&+\l|\int_{\mathbb{T}}(a^m(\eta)u(\eta))''v(\eta)d\eta\r|+\l|\int_{\mathbb{T}}(b^m(\eta)u(\eta))'v(\eta)d\eta\r|.
\end{split}
\end{equation}
For the first term of (\ref{10}),
\[
\begin{split}
\l|\int_{\mathbb{T}}\l[\int_{\mathbb{R}}c(\eta-q)\lambda^m(q)u(q)dq\r] v(\eta)d\eta-\int_{\mathbb{T}}\l[\int_{\mathbb{R}}c(\eta-q)\lambda^m(\eta)u(\eta))dq\r] v(\eta)d\eta\r|\\
\leq\l[\int_{\mathbb{T}}\l(\int_{\mathbb{R}}c(\eta-q)\lambda^m(q)u(q)dq\r)^2 d\eta\r]^{\frac{1}{2}}
\l(\int_{\mathbb{T}}v(\eta)^2d\eta\r)^{\frac{1}{2}}+a_1C_2\int_{\mathbb{T}}u(\eta)v(\eta)d\eta.\\
\end{split}
\]
In fact,
\begin{equation}
\begin{split}
&\int_{\mathbb{T}}\l(\int_{\mathbb{R}}c(\eta-q)\lambda^m(q)u(q)dq\r)^2d\eta\\
&=\int_{\mathbb{T}}\l(\int_{\mathbb{R}}c(\eta-q)\lambda^m(q)u(q)dq\r)
\l(\int_{\mathbb{R}}c(\eta-q)\lambda^m(q)u(q)dq\r)d\eta\\
&\leq\frac{C^2_2}{\delta^2}\int_{\mathbb{R}}c(q)dq\int_{\mathbb{R}}c(q)dq\int_{\mathbb{T}}u(\eta+q)u(\eta+q)
d\eta\leq a_2^2\frac{C^2_2}{\delta^2}\l\| u\r\|_0^2.
\end{split}
\end{equation}
Combining with (\ref{10}), we   conclude that
$$\l| a[u,v]\r|\leq C_3 \Vert u\Vert_0\Vert v\Vert_0+C_4 (\Vert u\Vert_0\Vert v'\Vert_0
+\Vert u'\Vert_0\Vert v'\Vert_0)+C_5\Vert u\Vert_0\Vert v'\Vert_0$$
$$\leq\nu\Vert u\Vert_1\Vert v\Vert_1.$$
We now use the assumptions to infer that
\begin{equation}\label{11}
\begin{split}
\kappa_1\Vert u'\Vert_0&\leq-\int_{\mathbb{T}}(a(\eta)u(\eta))''u(\eta)d\eta=a[u,u]-\int_{\mathbb{T}}(b(\eta)u(\eta))'u(\eta)d\eta\\
&\quad+\int_{\mathbb{T}}\l[\int_{\mathbb{R}}c(\eta-q)\big(\lambda(q)u(q)-\lambda
(\eta)u(\eta)\big)dq\r] u(\eta)d\eta\\&\leq a[u,u]+\int_{\mathbb{T}}\l(\Vert b\Vert_{\infty}\vert u'\vert \cdot\vert u\vert+C_7\vert u\vert^2\r)d\eta.
\end{split}
\end{equation}
Now we make use of the Young's inequality
$$ab\leq\delta_1 a^2+\frac{1}{4\delta_1}b^2,\; \text{for \;every}  \;\delta_1>0.$$
Using this in the second term on the right hand side of (\ref{11}), we obtain
$$\int_{\mathbb{T}}\vert u'\vert\cdot\vert u\vert d\eta\leq \delta_1 \Vert u'\Vert_0^2+\frac{1}{4\delta_1}
\Vert u\Vert_0^2.$$
We choose $\delta_1$, so that $$\kappa_1-\Vert b\Vert_{\infty}\delta_1=\frac{\kappa_1}{2}.$$
Thus $$\frac{\kappa_1}{2}\Vert u'\Vert_0^2\leq a[u,u]+\frac{1}{4\delta_1}\Vert b\Vert_{\infty}
\Vert u\Vert_0^2+C_7\Vert u\Vert_0^2.$$
We now add $\frac{\kappa_1}{2}\Vert u\Vert_0^2$ on the both sides of the preceding inequality to obtain
$$\frac{\kappa_1}{2}\Vert u\Vert_1^2\leq a[u,u]+\mu \Vert u\Vert_0^2,$$ with
$$\mu=\frac{1}{4\delta_1}\Vert b\Vert_{\infty}+C_7+\frac{\kappa_1}{2}.$$

Next we consider the resolvent operator$$R_{\l(\tilde{T}_m\r)^*}(\lambda)=\l((\tilde{T}_m)^*+\lambda I\r)^{-1},$$
where $I$ stands for the identity operator and $\lambda>0$.  Note that this operator is compact. For $\lambda$ sufficiently large, consequently, Fredholm theorem can be used for $R_{\l(\tilde{T}_m\r)^*}(\lambda)$. From the fact that the Fredholm alternative for $R_{\l(\tilde{T}_m\r)^*}(\lambda)$
implies the Fredholm alternative for $\tilde{T}_m^*,$ Fredholm theorem can be used for $\tilde{T}_m^*$(Lemma $7.11$ of \cite{Pav}).
Moreover, it is easy to see that $Ker \l(\tilde{T}_m\r)^*=\{C\}$, where $C$ is a constant.
Then we want to show the solvability condition:
\begin{equation}\label{13}
\int_{\mathbb{T}}l(\eta)d\eta=0.
\end{equation}
We take $z=\eta-q.$ Noting the fact that
$$\int_{\mathbb{R}}\int_{\mathbb{T}}q c(\eta-q)m(\eta)\lambda(\eta)d\eta dq=
\int_{\mathbb{R}}\int_{\mathbb{T}}\eta c(\eta-q)m(q)\lambda(q)d\eta dq,$$
we infer that
\[
\begin{split}
\int_{\mathbb{T}}l(\eta)d\eta&=\int_{\mathbb{R}}\int_{\mathbb{T}}qc(\eta-q)\big(m(\eta)\lambda(\eta)-m(q)\lambda(q)\big)d\eta dq\\
&\quad+\int_{\mathbb{T}}b(\eta)m(\eta)d\eta-\int_{\mathbb{T}}2(a(\eta)m(\eta))'d\eta\\
&=\int_{\mathbb{T}}\eta\int_{\mathbb{R}}c(q-\eta)\big(m(q)\lambda(q)-m(\eta)\lambda(\eta)\big)dqd\eta\\
&\quad+\int_{\mathbb{T}}b(\eta)m(\eta)d\eta-\int_{\mathbb{T}}2(a(\eta)m(\eta))'d\eta\\
&=\int_{\mathbb{T}}\eta\l[-(a(\eta)m(\eta))''+(b(\eta)m(\eta))'\r]d\eta\\
&\quad+\int_{\mathbb{T}}b(\eta)m(\eta)d\eta-\int_{\mathbb{T}}2(a(\eta)m(\eta))'d\eta\\
&=0.
\end{split}
\]
The solvability condition $\int_{\mathbb{T}}l_1(\eta)d\eta$ will be verified in the Appendix C.
Thus, the solution $h_1(\eta)$ and $h_2(\eta)$ is existence and uniqueness.
\end{proof}
\section*{Appendix C: Proof of Lemma \ref{2.11}}
\begin{proof}
Substituting $\xi_\epsilon$ defined in $(\ref{15})$ into $(T^\epsilon)^*\xi_\epsilon:$
   \begin{small}
  \[
  \begin{split}
  (T^\epsilon)^*(\xi^\epsilon)(x)&=\frac{1}{\epsilon^3}\int_{\mathbb{R}}c\l(\frac{x-y}{\epsilon}\r)
  \bigg\{\lambda\l(\frac{y}{\epsilon}\r)m\l(\frac{y}{\epsilon}\r)\l[\xi(y)\!+\!\epsilon h_{1}\l(\frac{y}{\epsilon}\r)\xi'(y)\!+\!\epsilon^2 h_2 \l(\frac{y}{\epsilon}\r)\xi''(y)\r]\\
  &\quad-\lambda\l(\frac{x}{\epsilon}\r)m\l(\frac{x}{\epsilon}\r)\l[\xi(x)+\epsilon h_{1}\l(\frac{x}{\epsilon}\r)\xi'(x)+\epsilon^2 h_2 \l(\frac{x}{\epsilon}\r)\xi''(x)\r]\bigg\}\\
  &\quad+\l\{a^m\l(\frac{x}{\epsilon}\r)\l[\xi(x)+\epsilon h_{1}\l(\frac{x}{\epsilon}\r)\xi'(x)+\epsilon^2 h_2 \l(\frac{x}{\epsilon}\r)\xi''(x)\r]\r\}''\\
 &\quad-\frac{1}{\epsilon}\l\{b^m\l(\frac{x}{\epsilon}\r)\l[\xi(x)+\epsilon h_{1}\l(\frac{x}{\epsilon}\r)\xi'(x)+\epsilon^2 h_2 \l(\frac{x}{\epsilon}\r)\xi''(x)\r]\r\}'dy.
\end{split}
\]
\end{small}
   First of all, we consider the term $(B^\epsilon)^*(\xi^\epsilon)(x)$ ,
  \[
  \begin{split}
    (B^\epsilon)^*(\xi^\epsilon)(x)&=\frac{1}{\epsilon^2}\int_{\mathbb{R}}c(z)\bigg\{\lambda\l(\frac{x}{\epsilon}-z\r)m\l(\frac{x}{\epsilon}-z\r)\Big[\xi(x-\epsilon z)\\
    &\quad+\epsilon h_{1}\l(\frac{x}{\epsilon}-z\r)\xi'(x-\epsilon z)
    +\epsilon^2 h_2\l(\frac{x}{\epsilon}-z\r)\xi''(x-\epsilon z)\Big]\\
    &\quad-\lambda\l(\frac{x}{\epsilon}\r)m\l(\frac{x}{\epsilon}\r)\l[\xi(x)+\epsilon h_{1}\l(\frac{x}{\epsilon}\r)\xi'(x)+\epsilon^2 h_2 \l(\frac{x}{\epsilon}\r)\xi''(x)\r]\bigg\}dz.
  \end{split}
   \]
Using the following identities based on the integral form of remainder term in the Taylor expansion
$$\xi(y)=\xi(x)+\int_{0}^{1}\frac{\partial}{\partial t}\xi(x+(y-x)t)dt=\xi(x)+\int_{0}^{1}\xi'(x+(y-x)t)\cdot(y-x)dt,$$
and
$$\xi(y)=\xi(x)+\xi'(x)(y-x)+\int_{0}^{1}\xi''(x+(y-x)t)(y-x)\cdot(y-x)(1-t)dt,$$
which is valid for each $x,y\in\mathbb{R}$, we conclude that
\[
\begin{split}
(B^\epsilon)^*(\xi^\epsilon)(x)&=
\frac{1}{\epsilon^2}\int_{\mathbb{R}}c(z)
\bigg\{\lambda\l(\frac{x}{\epsilon}-z\r)m\l(\frac{x}{\epsilon}-z\r)\bigg[\xi(x)-\epsilon z\xi'(x)\\
&\quad+\epsilon^2\int_{0}^{1}\xi''(x-\epsilon zt)\cdot z^2 (1-t)dt+\epsilon h_{1}\l(\frac{x}{\epsilon}-z\r)\Big(\xi'(x)-\epsilon z\xi''(x)\\
&\quad+\epsilon^2\int_{0}^{1}\xi'''(x-\epsilon zt)\cdot z^2 (1-t)dt\Big)+\epsilon^2 h_2 \l(\frac{x}{\epsilon}-z\r)\xi''(x-\epsilon z)\bigg]\\
&\quad-\lambda\l(\frac{x}{\epsilon}\r)m\l(\frac{x}{\epsilon}\r)\l[\xi(x)+\epsilon h_{1}\l(\frac{x}{\epsilon}\r)\xi'(x)+\epsilon^2 h_2 \l(\frac{x}{\epsilon}\r)\xi''(x)\r]\bigg\}dz.
\end{split}
\]
Collecting  the equal power  terms with $(A^\epsilon)^* \xi^\epsilon$, we obtain
\begin{small}
\begin{equation}\label{365}
\begin{split}
&(T^\epsilon)^*(\xi^\epsilon)(x)\\
&= \frac{1}{\epsilon^2}\xi(x)\bigg\{\int_{\mathbb{R}}c(z)\bigg[\lambda(\frac{x}{\epsilon}-z)m(\frac{x}{\epsilon}-z)-\lambda(\frac{x}{\epsilon})m(\frac{x}{\epsilon})\bigg]dz+\l(am\r)''(\frac{x}{\epsilon})-\l(bm\r)'(\frac{x}{\epsilon})\bigg\}\\
&\quad+\frac{1}{\epsilon}\xi'(x)\bigg\{\int_{\mathbb{R}}c(z)\bigg[\l(-z+h_1(\frac{x}{\epsilon}-z)\r)\lambda(\frac{x}{\epsilon}-z)m(\frac{x}{\epsilon}-z)-\lambda(\frac{x}{\epsilon})m(\frac{x}{\epsilon})h_1(\frac{x}{\epsilon})\bigg]dz\\
&\quad+2\l(am\r)'(\frac{x}{\epsilon})+\l(amh_{1}\r)''(\frac{x}{\epsilon})-b(\frac{x}{\epsilon})m(\frac{x}{\epsilon})-\l(bmh_{1}\r)'(\frac{x}{\epsilon})\bigg\}\\
&\quad\!+\!\xi''(x)\bigg\{\int_\mathbb{R}c(z)\bigg[\lambda(\frac{x}{\epsilon}\!-\!z)m(\frac{x}{\epsilon}\!-\!z)(\frac{1}{2}z^2\!-\!zh_1(\frac{x}{\epsilon}\!-\!z)+h_2(\frac{x}{\epsilon}-z))\\
&\quad-\lambda(\frac{x}{\epsilon})m(\frac{x}{\epsilon})h_2(\frac{x}{\epsilon})\bigg]dz
+a(\frac{x}{\epsilon})m(\frac{x}{\epsilon})+2(amh_1)'(\frac{x}{\epsilon})+(amh_2)''(\frac{x}{\epsilon})-(bmh_2)'(\frac{x}{\epsilon})\\
&\quad-h_1(\frac{x}{\epsilon})b(\frac{x}{\epsilon})m(\frac{x}{\epsilon})\bigg\}
+\phi_\epsilon(x),
\end{split}
\end{equation}
\end{small}
with
\begin{small}
\[
\begin{split}
\phi_\epsilon(x)&=\frac{1}{\epsilon^2}\int_\mathbb{R}c(z)\bigg\{\epsilon^2\int_0^1\lambda(\frac{x}{\epsilon}-z)m(\frac{x}{\epsilon}-z)\xi''(x-\epsilon zt)z^2(1-t)dt\\&\quad- \frac{\epsilon^2}{2}\lambda(\frac{x}{\epsilon}-z)
m(\frac{x}{\epsilon}-z)\xi''(x)z^2
+\epsilon^3 h_1(\frac{x}{\epsilon}-z)\int_0^1 \xi'''(x-\epsilon zt)z^2(1-t)dt\\&\quad- \epsilon^3 h_2(\frac{x}{\epsilon}-z)\int_0^1\xi'''(x-\epsilon zt)zdt\bigg\}dz.
\end{split}
\]
\end{small}


Denote $\eta=\frac{x}{\epsilon}$ a variable on the period: $\eta\in \mathbb{T}$. we collect all the terms of the order $\epsilon^{-2}$ in (\ref{365}) and equate them to $0.$
\begin{small}
 $$\int_{\mathbb{R}}c(z)\big[\lambda(\eta-z)m(\eta-z)-\lambda(\eta)m(\eta)\big]dz
 +(a(\eta)m(\eta))''-(b(\eta)m(\eta))'=(\tilde{T})^*m(\eta)=0.$$
\end{small}
From the fact that $(\tilde{T}_m)^*(h_1)(\eta)=l(\eta)$, for the terms of order $\epsilon^{-1}$, we have
\begin{small}
\begin{equation}
\begin{array}{rl}\label{29}
0=&\displaystyle\int_{\mathbb{R}}c(z)\Big[\big(-z+h_1(\eta-z)\big)\lambda(\eta-z)m(\eta-z)-\lambda(\eta)m(\eta)h_1(\eta))\Big]dz\\[2ex]&+2(a(\eta)m(\eta))'
+\l(a(\eta)m(\eta)h_{1}(\eta)\r)''-b(\eta)m(\eta)-(b(\eta)m(\eta)h_{1}(\eta))'.
\end{array}
\end{equation}
\end{small}

At last, we collect the term of the order $\varepsilon^0.$
Our goal is to find the function $h_2$, such that the sum of these terms will be equal to $T^0\xi=Q \xi''$ with $Q>0$.
Then we have
\begin{equation}\label{555}
\begin{split}
&(\tilde{T}_m)^*(h_2)(\eta)\\
&=-Q+\int_{\mathbb{R}}c(z)\lambda(\eta-z)m(\eta-z)\l[\frac{1}{2}z^2-zh_1(\eta-z)\r]dz
+a(\eta)m(\eta)\\
&\quad+2(a(\eta)m(\eta)h_1(\eta))'-b(\eta)m(\eta)h_1(\eta).
\end{split}
\end{equation}


Similar to the equality (\ref{13}). In order to ensure the uniqueness of the function $h_2$, we see that  $Q$ is determined from the following solvability condition for equation (\ref{13})
\begin{equation}\label{08908}
\begin{split}
Q&=\int_{\mathbb{T}}\int_{\mathbb{R}}c(z)\lambda(\eta-z)m(\eta-z)\l[\frac{1}{2}z^2-zh_1(\eta-z)\r]dzd\eta\\
&\quad +\int_{\mathbb{T}}a(\eta)m(\eta)d\eta-\int_{\mathbb{T}}b(\eta)m(\eta)h_1(\eta)d\eta.
\end{split}
\end{equation}

Next, let's simplify the expression of $Q.$ A short calculation revealed that
\[
\begin{split}
&\int_{\mathbb{T}}\int_{\mathbb{R}}c(z)\lambda(\eta-z)m(\eta-z)zh_1(\eta-z)dzd\eta\\
&=\int_{\mathbb{T}}\int_{\mathbb{R}}(\eta-q)c(\eta-q)\lambda(q)m(q)h_1(q)dqd\eta\\
&=\int_{\mathbb{T}}\int_{\mathbb{R}}c(q-\eta
)(q-\eta)\lambda(\eta)m(\eta)h_1(\eta)dqd\eta\\
&=\int_{\mathbb{T}}\l[\int_{\mathbb{R}}c(\eta-q)(q-\eta)dq\r]\lambda(\eta)m(\eta)h_1(\eta)d\eta\\
&=0
\end{split}
\]
For the third term of $Q,$
\[
\begin{split}
\int_{\mathbb{T}}-b^m(\eta)h_1&(\eta)d\eta=\int_{\mathbb{T}}\tilde{T}_m\chi(\eta)h_1(\eta)d\eta=\int_{\mathbb{T}}\chi(\eta)(\tilde{T}_m)^*h_1(\eta)d\eta\\
&=\int_{\mathbb{T}}\chi(\eta)\l[\int_{\mathbb{R}}zc(z)m(\eta-z)\lambda(\eta-z)dz+b^m(\eta)-2(a^m(\eta))'\r]d\eta,
\end{split}
\]
and
\[
\begin{split}
\int_{\mathbb{T}}\chi&(\eta)b^m(\eta)d\eta=-\int_{\mathbb{T}}\chi(\eta)\tilde{T}_m\chi(\eta)d\eta\\
&=\int_{\mathbb{T}}a^m(\eta)(\chi'(\eta))^2d\eta+\frac{1}{2}\int_{\mathbb{T}}\int_{\mathbb{R}}\lambda^m(\eta)c(z)\l[\chi(\eta-z)-\chi(\eta)\r]^2dzd\eta\\
&=\int_{\mathbb{T}}a^m(\eta)(\chi'(\eta))^2d\eta+\frac{1}{2}\int_{\mathbb{T}}\int_{\mathbb{R}}\lambda^m(q)c(\eta-q)\l[\chi(q)-\chi(\eta)\r]^2dqd\eta,\\
\end{split}
\]
we conclude  that
\begin{small}

  \begin{equation*}
    Q\!=\!\int_{\mathbb{T}}a(\eta)m(\eta)(\chi'(\eta)\!+\!1)^2d\eta\!+\!\frac{1}{2}\int_{\mathbb{T}}\int_{\mathbb{R}}c(\eta\!-\!q)\lambda(q)m(q)\l[(\eta\!-\!q)\!+\!(\chi (\eta)\!-\!\chi (q))\r]^2d\eta dq.
  \end{equation*}
\end{small}

Our last step is to show that $\l\|\phi_\varepsilon(x)\r\|_0$ is vanishing as $\epsilon\rightarrow 0.$

Choose the term of order $\epsilon^0$ of $\phi_\epsilon(x)$,  and denote it  by $\phi_{\epsilon}^{(1)}(x)$. For an arbitrary positive constant $M$,  we infer that
  \begin{small}
  \[
\begin{split}
\phi_{\epsilon}^{(1)}(x)&\!=\!\frac{1}{\epsilon^2}\l[\int\limits_{\l\{|z|\!\leq \!M\cup|z|\!>\!M\r\}}c(z)\epsilon^2\lambda(\frac{x}{\epsilon}-z)m(\frac{x}{\epsilon}-z)
\int_0^1(
\xi''\l(x-\epsilon zt)-\xi''(x)\r)z^2(1-t)dt\r]dz\\
&:=\phi_{\epsilon}^{(2)}(x)+\phi_{\epsilon}^{(3)}(x).
\end{split}
\]
\end{small}
Then
$$\l\|\phi_{\epsilon}^{(2)}\r\|_0\leq\frac{C_2}{2\delta}\sup_{|z|\leq M}\l\|\xi''(x-\epsilon zt)-\xi''(x)\r\|_0\int_{\mathbb{R}}z^2c(z)dz,$$
$$\|\phi_{\epsilon}^{(3)}\|_0\leq \frac{2C_2}{\delta}\l\|\xi''(x)\r\|_0
\int_{|z|>M}z^2c(z)dz.$$
If we take $M=\frac{1}{\sqrt{\epsilon}}$, then
$$\l\|\phi_{\epsilon}^{(2)}\r\|_0\rightarrow 0 \quad and \quad  \l\|\phi_{\epsilon}^{(3)}\r\|_0\rightarrow 0, \quad as\quad \epsilon\rightarrow 0.$$
This implies that $$\l\|\phi_{\epsilon}^{(1)}\r\|_0\rightarrow 0, \quad \epsilon\rightarrow 0.$$
 For the second term of $\phi_\epsilon(x),$
 $$\phi_{\epsilon}^{(4)}(x)=\epsilon \int_{\mathbb{R}}c(z)h_1\l(\frac{x}{\epsilon}-z\r)\l[\int_0^1 \xi'''(x-\epsilon zt)z^2(1-t)dt\r]dz,$$
 we have
\begin{equation}\label{9}
 \l\|\phi_{\epsilon}^{(4)}(x)\r\|_0\leq \frac{\epsilon C_2}{2\delta}\sup_{z,q\in\mathbb{R}}\l\|h_1\l(\frac{x}{\epsilon}-z\r)
 \xi'''(x-\epsilon z+q)\r\|_0\int_{\mathbb{R}}z^2c(z)dz.
\end{equation}
Next, we estimate $$\sup_{z,q\in\mathbb{R}}\l\|h_1\l(\frac{x}{\epsilon}-z\r)
 \xi'''(x-\epsilon z+q)\r\|_0.$$
Taking $y=x-\epsilon z$, we deduce that
\[
\begin{split}
\sup_{q\in\mathbb{R}}\l\|h_1\l(\frac{y}{\epsilon}\r)
 \xi'''(y+q)\r\|_0&=\sup_{q\in\epsilon\mathbb{T}}\l\|h_1\l(\frac{y}{\epsilon}\r)
 \xi'''(y+q)\r\|_0\\
 &\leq\sup_{q\in\epsilon\mathbb{T}}
\displaystyle\sum_{k\in\mathbb{Z}}\int_{\epsilon k}^{\epsilon k+\epsilon}h_1\l(\frac{y}{\epsilon}\r)^2
[\xi'''(y+q)]^2dy\\
&\leq \l\|h_1\r\|^2_0\sum_{k\in\mathbb{Z}}\max_{y\in[\epsilon k,\epsilon k+\epsilon],q\in \epsilon\mathbb{T}}[\xi'''(y+q)]^2dy\\
&\rightarrow \l\|h_1\r\|^2_0\|\xi'''\|^2_0,
\end{split}
\]
as $\epsilon\rightarrow 0$.
Thus from (\ref{9}), it follows that $\l\|\phi_{\epsilon}^{(4)}\r\|_0\rightarrow 0$, as $\epsilon\rightarrow 0.$

Similarly,  for the third term, we have
\[
\begin{split}
\l\|\phi_{\epsilon}^{(5)}(x)\r\|_0&=\left\|\epsilon\int_{\mathbb{R}}dz c(z) h_2(\frac{x}{\epsilon}-z)\int_0^1\xi'''(x-\epsilon zt)zdt\right\|_0\\
&\rightarrow 0.
\end{split}
\]
In summary, we have $\l\|\phi_\epsilon(x)\r\|_0 \rightarrow 0,$ as $\epsilon\rightarrow 0.$
\\ \hspace*{\fill} \\
\end{proof}

\section*{Appendix D: Convergence of $(V^\epsilon)^*\xi^\epsilon$}
Here we will show the convergence of $(V^\epsilon)^*\xi^\epsilon,$ as $\epsilon$ goes to $0.$

First, Let's do a simple calculation for fractional Laplace operator. For every functions $f,g,\psi\in H^{\alpha/2}$,
 \begin{small}
 \begin{equation*}
 \begin{split}
   \big\langle(-&\Delta)^{\alpha/2}(f\cdot g)(x),\psi(x)\big\rangle=\int_{\mathbb{R}}\int_{\mathbb{R}}\big(f(x)g(x)-f(y)g(y)\big)\psi(x)\gamma^2(x,y)dxdy\\
   &=\frac{1}{2}\int_{\mathbb{R}}\int_{\mathbb{R}}\big(f(x)g(x)-f(y)g(y)\big)\big(\psi(x)-\psi(y)\big)\gamma^2(x,y)dxdy\\
   &=\frac{1}{2}\big(\mathcal{D^{*}}(fg)(x,y),\mathcal{D^{*}}\psi(x,y)\big)_{L^2{(\mathbb{R}\times\mathbb{R})}}\\
   &=\small{\frac{1}{2}\int_{\mathbb{R}}\int_{\mathbb{R}}\big[(f(x)-f(y)g(x)+f(x)(g(x)-g(y))\big](\psi(x)-\psi(y))\gamma^2(x,y)dxdy}\\
   &=\frac{1}{2}\Big(\mathcal{D^{*}}(f)(x,y)g(x)+\mathcal{D^{*}}(g)(x,y)f(y),\mathcal{D^{*}}\psi(x,y)\Big)_{L^2{(\mathbb{R}\times\mathbb{R})}}.\\
 \end{split}
 \end{equation*}
 \end{small}

  For the operator $L^\epsilon,$ we have
  \begin{small}
  \begin{equation*}
 \begin{split}
 &\big\langle (L^\epsilon)^*\xi^\epsilon,\psi\big\rangle=\bigg(-(-\Delta)^{\alpha/2}(\delta_1^\epsilon\xi^\epsilon)(x),\psi(x)\bigg)
 -\frac{1}{\epsilon^{\alpha-1}}\bigg((p^\epsilon(x)\xi^\epsilon(x))',\psi(x)\bigg)\\
 &=-\frac{1}{2}\bigg(\mathcal{D^{*}}(\delta_1^\epsilon\xi^\epsilon)(x,y),\mathcal{D^{*}}\psi(x,y)\bigg)_{L^2{(\mathbb{R}\times\mathbb{R})}}
 -\frac{1}{\epsilon^{\alpha-1}}\bigg((p^\epsilon(x)\xi^\epsilon(x))',\psi(x)\bigg)\\
 &=-\frac{1}{2}\bigg(\mathcal{D^{*}}(\delta_1^{m_1,\epsilon}\xi)(x,y),\mathcal{D^{*}}\psi(x,y)\bigg)_{L^2{(\mathbb{R}\times\mathbb{R})}}
 -\frac{\epsilon}{2}\bigg(\mathcal{D^{*}}(\delta_1^{m_1,\epsilon} h_{3}^\epsilon\xi^{'})(x,y),\mathcal{D^{*}}\psi(x,y)\bigg)_{L^2{(\mathbb{R}\times\mathbb{R})}}\\
 &\quad-\bigg(\epsilon^{-\alpha}(pm_1)'\l(\frac{x}{\epsilon}\r)+\epsilon^{1-\alpha}p^{m_1,\epsilon}(x)\xi'(x)+\epsilon^{2-\alpha}\Big(p^{m_1,\epsilon}(x)h_3^{\epsilon}(x)\Big)'\xi'(x),\psi(x)\bigg)\\
 &\quad-\bigg(\epsilon^{2-\alpha}p^{m_1,\epsilon}(x)h_3^{\epsilon}(x)\xi''(x),\psi(x)\bigg)\\
 &:=G_1+G_2-\bigg(\epsilon^{-\alpha}(pm_1)'\l(\frac{x}{\epsilon}\r)\xi(x)+\epsilon^{1-\alpha}p^{m_1,\epsilon}(x)\xi'(x)+\epsilon^{1-\alpha}(pm_1h_3)'\l(\frac{x}{\epsilon}\r)\xi'(x),\psi(x)\bigg)\\
 &\quad-\bigg(\phi^1_{\epsilon},\psi(x)\bigg),
  \end{split}
 \end{equation*}
 \end{small}
 where $(\phi^1_{\epsilon},\psi(x))\rightarrow 0,$ as $\epsilon$ goes to $0.$ Furthermore, we can show that,
  \begin{small}
  \begin{equation*}
 \begin{split}
 G_1&=-\frac{1}{2}\bigg(\mathcal{D^{*}}(\delta_1^{m_1,\epsilon}\xi)(x,y),\mathcal{D^{*}}\psi(x,y)\bigg)_{L^2{(\mathbb{R}\times\mathbb{R})}}\\
 &=-\frac{1}{2}\bigg(\mathcal{D^{*}}(\delta_1^{m_1,\epsilon})(x,y)\xi(x)+\mathcal{D^{*}}(\xi)(x,y)\delta_1^{m_1,\epsilon}(y),\mathcal{D^{*}}\psi(x,y)\bigg)_{L^2{
 (\mathbb{R}\times\mathbb{R})}}\\
 &=-\frac{1}{2}\bigg(\mathcal{D^{*}}(\delta_1^{m_1,\epsilon})(x,y),\mathcal{D^{*}}(\psi\xi)(x,y)-\mathcal{D^{*}}(\xi)(x,y)\psi(y)\bigg)_{L^2{(\mathbb{R}\times\mathbb{R})}}\\
 &-\frac{1}{2}\bigg(\mathcal{D^{*}}(\xi)(x,y),\delta_1^{m_1,\epsilon}(y)\mathcal{D^{*}}\psi(x,y)\bigg)_{L^2{(\mathbb{R}\times\mathbb{R})}}
 =\bigg(-(-\Delta)^{\alpha/2}\delta_1^{m_1,\epsilon}(x),\psi(x)\xi(x)\bigg)\\
 &-\frac{1}{2}\bigg(\mathcal{D^{*}}(\xi)(x,y),
 \delta_1^{m_1,\epsilon}(y)\mathcal{D^{*}}\psi(x,y)-\mathcal{D^{*}}\delta_1^{m_1,\epsilon}(x,y)\psi(y)\bigg)_{L^2{(\mathbb{R}\times\mathbb{R})}}\\
 &:=I_1+I_2.
  \end{split}
 \end{equation*}
 \end{small}
 For $I_2,$ we deduce that,
   \begin{equation*}
 \begin{split}
   I_2&=-\frac{1}{2}\bigg(\mathcal{D^{*}}(\xi)(x,y),\delta_1^{m_1,\epsilon}(y)\mathcal{D^{*}}\psi(x,y)-\mathcal{D^{*}}\delta_1^{m_1,\epsilon}(x,y)\psi(y)\bigg)
   _{L^2{(\mathbb{R}\times\mathbb{R})}}\\
   &=-\frac{1}{2}\bigg(\mathcal{D^{*}}(\xi)(x,y),\mathcal{D^{*}}(\delta_1^{m_1,\epsilon}\psi)(x,y)-\mathcal{D^{*}}\delta_1^{m_1,\epsilon}(x,y)\l[\psi(x)+\psi(y)\r]\bigg)
    _{L^2{(\mathbb{R}\times\mathbb{R})}}\\
   &:=\bigg(-(-\Delta)^{\alpha/2}\xi(x),\delta_1^{m_1,\epsilon}(x)\psi(x)\bigg)+I_3,\\
   \end{split}
 \end{equation*}
 where we have $I_3\rightarrow 0$. In fact,
  \begin{equation*}
 \begin{split}
   I_3&=\frac{1}{2}\bigg(\mathcal{D^{*}}(\xi)(x,y),\mathcal{D^{*}}\delta_1^{m_1,\epsilon}(x,y)\l[\psi(x)+\psi(y)\r]\bigg)_{L^2{(\mathbb{R}\times\mathbb{R})}}\\
   &=\int_{\mathbb{R}}\int_{\mathbb{R}}\l[\xi(x)-\xi(y)\r]\delta_1^{m_1,\epsilon}(x)\l[\psi(x)+\psi(y)\r]\gamma^2(x,y)dxdy\\
   &=\bigg(\int_{\mathbb{R}}\l[\xi(x)-\xi(y)\r]\l[\psi(x)+\psi(y)\r]\gamma^2(x,y)dy,\delta_1^{m_1,\epsilon}(x)\bigg)
   \rightarrow 0.
     \end{split}
 \end{equation*}
 From the calculation above, we can obtain that,
 \begin{equation*}
   \begin{split}
   G_2=\epsilon\bigg(-(-\Delta)^{\alpha/2}(\delta_1^{m_1,\epsilon} h_3^\epsilon)(x),\psi(x)\xi'(x)\bigg)-\bigg(\phi^2_\epsilon,\psi(x)\xi'(x)\bigg),
   \end{split}
 \end{equation*}
 where $\bigg(\phi_2^\epsilon,\psi(x)\xi'(x)\bigg)\rightarrow 0,$ as $\epsilon$ goes to $0.$
 Then, we have,
 \begin{small}
 \begin{equation*}
   \begin{split}
   \big\langle&(L^\epsilon)^*\xi^\epsilon,\psi\big\rangle=\bigg(-(-\Delta)^{\alpha/2}\delta_1^{m_1,\epsilon}(x)-\epsilon^{-\alpha}(pm_1)'\l(\frac{x}{\epsilon}\r),\psi(x)\xi(x)\bigg)\\
    &+\bigg(-\epsilon(-\Delta)^{\alpha/2}(\delta_1^{m_1,\epsilon} h_3^\epsilon)(x)-\epsilon^{1-\alpha}p^{m_1,\epsilon}(x)-\epsilon^{1-\alpha}(pm_1h_3)'\l(\frac{x}{\epsilon}\r),\psi(x)\xi'(x)\bigg)\\
   &+\bigg(-(-\Delta)^{\alpha/2}\xi(x),
   \delta_1^{m_1,\epsilon}(x)\psi(x)\bigg)+\phi_\epsilon^3\\
   \end{split}
 \end{equation*}
 \end{small}
 where $\phi_\epsilon^3$ goes to $0.$
 Using the equations (\ref{25}), (\ref{38}), as $\epsilon\rightarrow 0$, we have
 $$\l\langle(L^\epsilon)^*\xi^\epsilon,\psi\r\rangle\rightarrow \l(\int_\mathbb{T}\delta^\alpha(\eta)m_1(\eta)d\eta\r)\cdot\bigg(-(-\Delta)^{\alpha/2}\xi(x),\psi(x)\bigg). $$

From the fact that $$\l\langle F^\epsilon)^*\xi^\epsilon,\psi\r\rangle\rightarrow\l(-\xi'(x)\int_{\mathbb{T}}g(\eta)m_1(\eta)d\eta+\xi(x)\int_{\mathbb{T}}f(\eta)m_1(\eta)d\eta,\psi(x)\r)$$

 We can infer that
 \[
 \begin{split}
 \l\langle (V^\epsilon)^*\xi^\epsilon,\psi\r\rangle\rightarrow& \bigg(\int_{\mathbb{T}}\delta^\alpha(\eta)m_1(\eta)d\eta\cdot\Big(-(-\Delta)^{\alpha/2}\Big)\xi(x)+\xi'(x)\int_{\mathbb{T}}g(\eta)m_1(\eta)d\eta\\
 &+\xi(x)\int_{\mathbb{T}}f(\eta)m_1(\eta)d\eta,\psi(x)\bigg).
 \end{split}
 \]
\end{appendices}

%
%



\end{document}